\documentclass[reqno]{amsart}
\usepackage{amsmath,amssymb,accents,amsthm}
\usepackage[latin1]{inputenc}
\allowdisplaybreaks[4]
\newtheorem{theorem}{Theorem}[section]
\newtheorem{proposition}[theorem]{Proposition}
\newtheorem{claim}[theorem]{Claim}
\newtheorem{definition}[theorem]{Definition}
\newtheorem{corollary}[theorem]{Corollary}
\newtheorem{lemma}[theorem]{Lemma}

\theoremstyle{definition}
\numberwithin{equation}{section}
\newcommand\bm{\begin{pmatrix}}
\renewcommand\em{\end{pmatrix}}
\renewcommand\({\left(}
\renewcommand\){\right)}

\renewcommand\[{\left[}
\renewcommand\]{\right]}
\newcommand\er{\eqref}
\newcommand\be[1]{\begin{equation}\label{#1}}
\newcommand\ee{\end{equation}}
\newcommand\I{\int\limits}
\renewcommand\S{\sum\limits}
\renewcommand\P{\prod\limits}

\newcommand\ab[2]{\begin{pmatrix}{#1}&{#2}\end{pmatrix}}
\newcommand\ba[2]{\begin{pmatrix}{#1}\\{#2}\end{pmatrix}}
\newcommand\bb[4]{\begin{pmatrix}{#1}&{#2}\\{#3}&{#4}\end{pmatrix}}

\newcommand\h[1]{{\hat{#1}}}

\renewcommand\t[1]{{\tilde{#1}}}
\newcommand\w[1]{{\widetilde{#1}}}
\renewcommand\o[1]{{\overline{#1}}}

\newcommand\F[1]{{\sqrt{#1}}}
\newcommand\f[2]{{\frac{#1}{#2}}}

\newcommand\la{\alpha}
\newcommand\lb{\beta}
\renewcommand\lg{\gamma}
\newcommand\lG{\Gamma}
\newcommand\ld{\delta}
\newcommand\lD{\Delta}
\newcommand\Le{\epsilon}
\newcommand\lh{\eta}

\renewcommand\ll{\lambda}
\newcommand\lL{\Lambda}
\newcommand\lm{\mu}
\renewcommand\ln{\nu}

\newcommand\lf{\phi}

\renewcommand\lq{\psi}

\newcommand\lp{\pi}

\newcommand\lr{\rho}
\newcommand\ls{\sigma}
\newcommand\lS{\Sigma}

\newcommand\lT{\Theta}
\newcommand\Lt{\tau}
\newcommand\lo{\omega}
\newcommand\lO{\Omega}
\newcommand\lx{\xi}
\newcommand\lz{\zeta}
\newcommand\ui{\cap}

\newcommand\ic{\subset}

\newcommand\xx{\times}
\newcommand\xt{\otimes}
\newcommand\oo{\infty}
\newcommand\ol{\bullet}
\newcommand\oc{\circ}
\newcommand\dl{\partial}

\newcommand\al{\approx}
\newcommand\el{\ell}
\newcommand\yi{\wedge}
\newcommand\sm{\setminus}

\newcommand\AL{\mathcal A}
\newcommand\BL{\mathcal B}
\newcommand\CL{\mathcal C}

\newcommand\EL{\mathcal E}

\newcommand\HL{\mathcal H}
\newcommand\IL{\mathcal I}
\newcommand\JL{\mathcal J}
\newcommand\KL{\mathcal K}
\newcommand\LL{\mathcal L}

\newcommand\PL{\mathcal P}

\newcommand\SL{\mathcal S}
\newcommand\TL{\mathcal T}

\newcommand\Bl{\mathbf B}
\newcommand\Cl{\mathbf C}

\newcommand\Nl{\mathbf N}
\newcommand\Ol{\mathbf O}
\newcommand\Pl{\mathbf P}
\newcommand\Rl{\mathbf R}
\newcommand\Sl{\mathbf S}

\newcommand\oL{\frak o}

\newcommand\bi{\begin{itemize}}
\newcommand\ei{\end{itemize}}

\begin{document}
\title{Dixmier Trace for Toeplitz Operators on Symmetric Domains}
\author{Harald Upmeier and Kai Wang}

\address{ Fachbereich Mathematik,
Universit$\ddot{a}$t Marburg, Marburg, 35032, Germany}
\email{upmeier@mathematik.uni-marburg.de}

\address{School of Mathematical Sciences,
Fudan University, Shanghai, 200433, P. R. China}
\email{kwang@fudan.edu.cn}

\thanks{The second  author was partially supported by NSFC (11271075,11420101001), the Alexander von Humboldt Foundation and Laboratory of Mathematics for Nonlinear Science at Fudan University.}

\maketitle

\begin{abstract}
For Toeplitz operators on bounded symmetric domains of arbitrary rank, we define a Hilbert quotient module corresponding to partitions of length $1$ and prove that it belongs to the Macaev class $\LL^{n,\oo}$. We next obtain an explicit formula for the Dixmier trace of Toeplitz commutators in terms of the underlying boundary geometry.\end{abstract}

\section{Introduction}
The Dixmier trace of Hilbert space operators \cite{C2}, of fundamental importance for pseudo-differential operators \cite{C1,W}, has recently found deep applications in {\bf complex analysis}, for Hankel and Toeplitz operators on strictly pseudo-convex domains \cite{AFJP,EGZ,EZ,ER,HH} and for homogeneous Hilbert quotient modules over the unit ball \cite{DTY,EE,GW,GWZ}. In these applications the underlying operators are essentially normal, i.e. commutators are compact; more precisely, belong to certain norm ideals of Schatten type.

In this paper we are concerned with operators of Toeplitz or Hankel type which are not essentially commuting. These operators arise naturally when the underlying domain $D\ic\Cl^d$ is not strictly pseudo-convex or does not have a smooth boundary. The most important case is the so-called {\bf hermitian bounded symmetric domains} $D=G/K$ of arbitrary rank $r$, which generalize the unit disk and the unit ball of rank $1.$ In this paper, we construct a suitable Hilbert quotient module of the Hardy space over the Shilov boundary $S$ and study the associated 'sub-Toeplitz' operators. Our first main result shows that commutators of such operators belong to the Macaev class $\LL^{n,\oo},$ for a suitable $n$ related to the geometry of $D.$ The second main result is an explicit formula for the Dixmier trace of products of such operators, in terms of a Jordan theoretic Grassmann-type manifold.

The results of this paper can be generalized to cover the weighted Bergman spaces instead of the Hardy space, at least for the continuous part of the Wallach set \cite{FK}. On the other hand, extending these results to all smooth functions $f\in\CL^\oo(S)$ will be more challenging, even for the basic case of rank $2$-domains (involving pseudo-differential operators on spheres \cite{BC,BCK}).
Finally, the higher strata of the boundary of $D$ give rise to a family of smooth extensions \cite{U5} and it is of interest to develop a family version of the Dixmier trace (involving cyclic cohomology) for the associated Toeplitz commutators.

\section{Symmetric domains and Toeplitz operators}
Let $D$ be an irreducible bounded symmetric domain of rank $r$ in a complex vector space $Z$ of finite dimension $d$. The unit ball $D=\Bl_d\ic\Cl^d$ corresponds to rank $r=1.$ Denote by $G=Aut(D)$ the biholomorphic automorphism group, and put
$$K:=\{g\in G: g(0)=0\}.$$
Then $D=G/K.$ It is well known \cite{FK,L} that $D$ can be realized as the open unit ball of an irreducible {\it hermitian Jordan triple} $Z$. Thus $Z$ is a complex vector space endowed with a Jordan triple product
$$u,v,w\mapsto\{uv^*w\}\in Z\qquad\forall\,u,v,w\in Z.$$
Then $K=Aut(Z)$ is the linear group of all triple automorphisms of $Z.$ Let $S$ be the Shilov boundary of $D$. Since $K$ acts transitively on $S,$ there exists a unique $K-$invariant probability measure $ds$ on $S$. Denote by $L^2(S)$ the space of $L^2$-integrable functions, with inner product
\be{5}(f|g)_S:=\I_{S}ds\,\o{f(s)}\,g(s),\ee
and define the {\bf Hardy space}
$$H^2(S)=\{\lq\in L^2(S):\,\lq\mbox{ holomorphic on }D\}.$$
For a bounded function $f$, define the {\bf Toeplitz operator}
$$T_f\lq= P_{H^2(S)}(f\lq)\qquad\forall\lq\in H^2(S).$$
In previous work \cite{U1,U2} it was shown that Toeplitz operators $T_f$ with smooth symbol function $f\in\CL^\oo(S)$, acting on
$H^2(S),$ generate a $C^*$-algebra $\TL(S)$ which is not essentially commutative (if $r>1$) but has a {\bf composition series}
$$\KL=\IL_1\ic\IL_2\ic\cdots\ic\IL_r\ic\TL(S)=\IL_{r+1},$$
starting with the compact operators $\KL,$ such that the subquotients $\IL_{k+1}/\IL_k$ are essentially commutative. More precisely,
there is a stable isomorphism
$$\IL_{k+1}/\IL_k\al\CL(S_k)\xt\KL,$$
where $S_k$ denotes the $K$-homogeneous manifold of all 'tripotents' of rank $k.$ (Similar results hold for Toeplitz operators on weighted Bergman spaces over $D,$ as shown in \cite{U4}.) An element $c\in Z$ such that $\{cc^*c\}=c$ is called a {\bf tripotent}. Every tripotent induces a {\bf Peirce decomposition}
$$Z=Z^2_c\oplus Z^1_c\oplus Z^0_c,$$
where $Z^\la_c:=\{z\in Z:\,\{cc^*z\}=2\la z\}.$
The Peirce $2$-space is a Jordan $*$-algebra with unit element $c$ and involution $z\mapsto \{cz^*c\}.$ The self-adjoint part $X_c\ic Z^2_c$ is a so-called {\bf euclidean Jordan algebra} \cite{FK}. Let $(z|w)$ denote the $K$-invariant inner product normalized by the condition $(c|c)=1$ for each minimal tripotent $c\in Z.$ Let $\PL(Z)$ be the algebra of all (holomorphic) polynomials on $Z,$ endowed with the $K$-invariant {\bf Fischer-Fock inner product}
\be{4}(p|q)_Z:=\f1{\lp^d}\I_{Z}dz\,e^{-(z|z)}\o{p(z)}\,q(z)\ee
for all $p,q\in\PL(Z).$ By \cite{FK,U3} the natural action of $K$ on $\PL(Z)$ induces a multiplicity-free {\bf Peter-Weyl decomposition}
\be{1} \PL(Z)=\S_\ll\PL_\ll(Z),\ee
where
$$\ll=\ll_1\ge\ldots\ge \ll_r\ge 0$$
runs over all integer {\bf partitions} of length $\le r.$ The decomposition \er{1} is orthogonal under \er{4}. We let $\Nl^r_+$ denote the set of all such partitions. As usual we will identify partitions that differ only by zeros. Then
$$\Nl^r_+=\bigcup_{\el=1}^r\Nl^\el_+,$$
where $\Nl^\el_+=\{\ll\in\Nl^r_+:\,\ll_{\el+1}=0\}.$ As a special type of partition we denote
$$k_\el:=(k,\ldots,k,0\ldots,0)$$
for $1\le\el\le r$ and $k\in\Nl$ repeated $\el$ times. Choose a frame $e_1,\ldots,e_r$ of minimal tripotents. The associated joint Peirce decomposition \cite{L} defines two numerical invariants $a,b$ for $Z$ such that
$$\lr:=\f dr=1+\f a2(r-1)+b.$$
For the Hilbert unit ball ($r=1$) we put $a=2$ and $b=d-1.$ Thus $b=0$ only for the unit disk. In case $b=0$ the Jordan triple $Z$ is actually a {\bf Jordan algebra} with unit element
$$e:=e_1+\cdots+e_r.$$
In this case $Z$ carries a {\bf Jordan determinant} $N=N_r$ which is normalized by $N(e)=1.$ For $1\le\el\le r$ denote by $N_\el$ the Jordan determinant polynomial for the Peirce $2$-space $Z^2_{e_1+\ldots+e_\el}.$ As shown in \cite{U3} $\PL_\ll(Z)$ has the highest weight vector
\be{2}N_\ll(z):=N_1(z)^{\ll_1-\ll_2}N_2(z)^{\ll_2-\ll_3}\cdots N_r(z)^{\ll_r}.\ee
The {\bf multi-variable Pochhammer symbol} is the product
$$(s)_\ll:=\P_{i=1}^r(s-\f a2(i-1))_{\ll_i}$$
of the usual Pochhammer symbols $(\ln)_m=\P_{i=1}^m(\ln+i-1).$ By \cite{U1,FK} the inner products \er{4} and \er{4} are related by
\be{30}(p|q)_S=\f1{(\lr)_\ll}(p|q)_Z,\qquad\forall\, p,q\in\PL_\ll(Z).\ee
We note the relation
$$\f{(\lr)_\ll}{(\lr-b)_\ll}=\P_{j=1}^r\f{(\ll_j+1+\f a2(r-j))_b}{(1+\f a2(r-j))_b}.$$
\begin{proposition}\label{nn} For $\ll\in\Nl^r_+$ we have
\be{12}\|N_\ll\|_S^2=\f{(\lr-b)_\ll}{(\lr)_\ll}\P_{1\le i<j\le r}\f{(1+\f a2(j-i-1))_{\ll_i-\ll_j}}{(1+\f a2(j-i))_{\ll_i-\ll_j}}.\ee
\end{proposition}
\begin{proof} Using the reciprocity relation
\be{6}\f{(x+b)_m}{(x)_m}=\f{(x+m)_b}{(x)_b}\ee
for integers $0\le b\le m,$ the assertion follows from \cite{U1} or (for tube domains) \cite[Proposition XI.4.3]{FK}.
\end{proof}
For any partition $\ll$ let
\be{7}P_\ll:\PL(Z)\to\PL_\ll(Z)\ee
denote the orthogonal projection. If $f_u(z)=(z|u)$ is a linear functional associated with $u\in Z,$ we simply write $T_u:=T_{f_u}.$
Moreover, $u^\dl$ denotes the directional derivative. By \cite[Theorem 2.11]{U1} we have
\be{16}T_u^*q=\S_i P_{\ll-[i]}T_u^*q=\S_{i=1}^r\f1{\ll_i+\f a2(r-i)+b}P_{\ll-[i]}u^\dl q\ee
for all $q\in\PL_\ll(Z),$ where
$$[i]=(0,\ldots,0,1,0,\ldots,0)$$
with $1$ at position $i.$ More precisely, only those terms occur where $\ll-[i]$ is again a partition.
\begin{definition} Let $\SL$ denote the set of all sequences
$$c_m=c_0+\f{c_1}{m+1}+\oL_m,$$
where $c_0,c_1$ are constants and the sequence $\{m^2\oL_m\}_m$ is bounded. Let $\SL_+$ denote the set of sequences in $\SL$ with $c_0>0$.
\end{definition}
It is clear that $\SL$ is closed under taking finite sums and products of sequences. $\SL_+$ is also closed under taking quotients.
\begin{lemma}\label{qn} Let $\la,\lg\in\Nl^{r-1}_+.$ Then $\left\{\f{\|N_{m-k,\lg}\|_S^2}{\|N_{m,\la}\|_S^2}\right\}_m\in\SL_+.$
\end{lemma}
\begin{proof} In terms of the falling Pochhammer symbol $(m)_j^*=\P_{i=1}^j(m+1-i),$ \er{12} implies
$$\f{\|N_{m-k,\lg}\|_S^2}{\|N_{m,\la}\|_S^2}=C\P_{j=1}^{r-1}\f{(m-\la_j+\f a2 j)_{k+\lg_j-\la_j}^*}{(m-\la_j+\f a2(j-1))_{k+\lg_j-\la_j}^*}$$
whenever $k+\lg_j\ge\la_j.$ Since each factor belongs to $\SL_+,$ the assertion follows.
\end{proof}
\begin{lemma}\label{i} Let $\el\le r$ and $\ll\in\Nl^\el_+.$ Then
$$T_{N_\el^k}^*N_\ll=\P_{j=1}^\el\f{(\ll_j+\f a2(\el-j))_k^*}{(\ll_j+\f a2(r-j)+b)_k^*}N_{\ll-k_\el},\qquad\forall\,k\le\ll_\el.$$
\end{lemma}
\begin{proof} Consider the Peirce $2$-space $\t Z=Z^2_{e_1+\ldots+e_\el}$ of rank $\el$ and put $\t\lr=1+\f a2(\el-1).$ Using
$N_\ll=N_\el^k\,N_{\ll-k_\el}$ and applying \er{30} to $S\ic Z$ and $\t S\ic\t Z,$ we obtain for $\lf\in\PL(\t Z)$
$$(\lr)_\ll(\lf|T_{N_\el^k}^*N_\ll)_S=(\lr)_\ll(N_\el^k\lf|N_\ll)_S=(N_\el^k\lf|N_\ll)_Z=(N_\el^k\lf|N_\ll)_{\t Z}
=(\t\lr)_\ll(N_\el^k\lf|N_\ll)_{\t S}$$
$$=(\t\lr)_\ll(\lf|N_{\ll-k_\el})_{\t S}=\f{(\t\lr)_\ll}{(\t\lr)_{\ll-k_\el}}(\lf|N_{\ll-k_\el})_{\t Z}=\f{(\t\lr)_\ll}{(\t\lr)_{\ll-k_\el}}(\lf|N_{\ll-k_\el})_Z=\f{(\t\lr)_\ll(\lr)_{\ll-k_\el}}{(\t\lr)_{\ll-k_\el}}(\lf|N_{\ll-k_\el})_S.$$
Since $\lf$ is arbitrary, it follows that
$$T_{N_\el^k}^*N_\ll=\f{(\t\lr)_\ll(\lr)_{\ll-k_\el}}{(\t\lr)_{\ll-k_\el}(\lr)_\ll}N_{\ll-k_\el}.$$
We have
\be{31}\f{(\lr)_\ll}{(\lr)_{\ll-k_\el}}=\P_{j=1}^\el(\ll_j+\f a2(r-j)+b)_k^*.\ee
Applying \er{31} to $Z$ and $\t Z,$ the assertion follows.
\end{proof}
We will now consider partitions $(m,0,\ldots,0)=m$ of length $1,$ with projection $P_m:H^2(S)\to\PL_m(Z).$ Here $\PL_m(Z)$ is spanned by the $K$-orbit of the conical polynomial $N_1^m.$  As shown in \cite{U2}, the projection
\be{8}P:=\S_m P_m\ee
on $H^2(S)$ belongs to the Toeplitz $C^*$-algebra $\TL(S).$ For a partition $\ll$ choose an orthonormal basis $p_i\in\PL_\ll(Z),$ for the inner product \er{4}. Then
$$A^\ll:=\S_i T_{p_i}PT_{p_i}^*$$
is a $K$-invariant operator, independent of the choice of orthonormal basis. Since the decomposition \er{1} is multiplicity-free, every $K$-invariant operator $T$ on $\PL(Z)$ (or $H^2(S)$) is a 'diagonal' operator. Define
$$\ll':=(\ll_2,\ldots,\ll_r)\in\Nl^{r-1}_+.$$
\begin{lemma}\label{z} Let $p\in\PL_\ll(Z)$ such that $PT_p^*N_{m,\lb}\ne 0.$ Then $\lb\le\ll\le(m,\lb).$
\end{lemma}
\begin{proof} For $\lf\in\PL_{n,0}(Z)$ the non-zero components of $p\lf$ correspond to signatures $\lm$ obtained from $\ll$ by adding a horizontal $n$-strip \cite[Proposition 5.3]{S2}. Thus
$$\lm'\le\ll\le\lm.$$
It follows that $(\lf|T_p^*N_{m,\lb})_S=(p\,\lf|N_{m,\lb})_S$ is non-zero only if $\lm=(m,\lb)$ satisfies the above condition, which leads to $\lb\le\ll\le(m,\lb).$
\end{proof}
Since
$$\mbox{Ran}(T_{p_i}P)\ic\S_{\lm'\le\ll\le\lm}P_\lm$$
by Lemma \er{z}, it follows that
\be{32}A^\ll=\S_{\lm'\le\ll\le\lm}\f{(N_\lm|A^\ll N_\lm)_S}{\|N_\lm\|_S^2}\,P_\lm,\ee
where
\be{cn}\f{(N_\lm|A^\ll N_\lm)_S}{\|N_\lm\|_S^2}=\f1{\|N_\lm\|_S^2}(N_\lm|\S_i T_{p_i}PT_{p_i}^*\,N_\lm)_S=\f1{\|N_\lm\|_S^2}\S_i\|PT_{p_i}^*N_\lm\|_S^2.\ee
\begin{proposition}\label{j} Let $\ll\in\Nl^r_+.$ Then
$\left\{\f{(N_{m,\ll'}|A^\ll N_{m,\ll'})_S}{\|N_{m,\ll'}\|_S^2}\right\}_m\in\SL_+.$
\end{proposition}
\begin{proof} The proof is by induction on the length $\el\le r$ of $\ll.$ Put $k:=\ll_\el>\ll_{\el+1}=0.$ Then
$\lg:=\ll-k_\el$ has length $<\el.$ Consider the Peirce $2$-space $\t Z:=Z^2_{e_1+\ldots+e_\el}$ of rank $\el.$ We may assume that a subfamily $p_i:i\in\t I$ is an orthonormal basis of $\PL_\ll(\t Z).$ Since $\PL_\ll(\t Z)=N_\el^k\,\PL_\lg(\t Z),$ there exists a constant $c>0$ such that $p_i=c\cdot N_\el^k\,q_i$ for all $i\in\t I,$ where $q_i\in\PL_\lg(\t Z)$ is an orthonormal basis. For
$m\ge\ll_2$ it follows from Lemma \er{i} that $T_{N_\el^k}^*N_{m,\ll'}=c_m N_{m-k,\lg'},$ where
$$c_m=\f{(m+\f a2(\el-1))_k^*}{(m+b+\f a2(r-1))_k^*}\P_{j=2}^\el\f{(\ll_j+\f a2(\el-j))_k^*}{(\ll_j+b+\f a2(r-1))_k^*}$$
belongs to $\SL_+,$ in view of the identity
$$\f{m+a}{m+b}=1+\f{a-b}m-\f{(a-b)b}{m(m+b)}.$$
For $i\in I\sm\t I$ we have $T_{p_i}^*N_{m,\ll'}=0$ since $p_i$ belongs to the ideal generated by $\t Z^\perp.$ It follows that
$$(N_{m,\ll'}|A^\ll N_{m,\ll'})_S=\S_{i\in I}(T_{p_i}^*N_{m,\ll'}|PT_{p_i}^*N_{m,\ll'})_S
=\S_{i\in\t I}(T_{p_i}^*N_{m,\ll'}|PT_{p_i}^*N_{m,\ll'})_S$$
$$=c^2\S_{i\in\t I}(T_{q_i}^*T_{N_\el^k}^*N_{m,\ll'}|PT_{q_i}^*T_{N_\el^k}^*N_{m,\ll'})_S
=c^2\cdot c_m^2\S_{i\in\t I}(T_{q_i}^*N_{m-k,\lg'}|PT_{q_i}^*N_{m-k,\lg'})_S.$$
Now consider the $K$-invariant operator
$$A^\lg=\S_{j\in J}T_{q_j}PT_{q_j}^*,$$
where $q_j,\,j\in J$ is an orthonormal basis of $\PL_\lg(Z).$ We may assume that $q_i,i\in\t I$ are a subfamily of $J.$
As above, we have $T_{q_j}^*N_{m-k,\lg'}=0$ whenever $j\in J\sm\t I.$ Therefore
$$A^\lg N_{m-k,\lg'}=\S_{j\in J}T_{q_j}PT_{q_j}^*N_{m-k,\lg'}=\S_{i\in\t I}T_{q_i}PT_{q_i}^*N_{m-k,\lg'}$$
and hence $(N_{m,\ll'}|A^\ll N_{m,\ll'})_S=c^2\cdot c_m^2(N_{m-k,\lg'}|A^\lg N_{m-k,\lg'})_S.$ Since $\lg$ has length $<\el,$ the induction hypothesis implies that $\left\{\f{(N_{m-k,\lg'}|A^\lg N_{m-k,\lg'})_S}{\|N_{m-k,\lg'}\|_S^2}\right\}_m\in\SL_+.$ It follows that the sequence
$$\f{(N_{m,\ll'}|A^\ll N_{m,\ll'})_S}{\|N_{m,\ll'}\|_S^2}
=c^2\,c_m^2\,\f{(N_{m-k,\lg'}|A^\lg N_{m-k,\lg'})_S}{\|N_{m-k,\lg'}\|_S^2}\,\f{\|N_{m-k,\lg'}\|_S^2}{\|N_{m,\ll'}\|_S^2}$$
belongs to $\SL_+,$ since Lemma \er{qn} implies that $\left\{\f{\|N_{m-k,\lg'}\|_S^2}{\|N_{m,\ll'}\|_S^2}\right\}_m\in\SL_+.$
\end{proof}

\section{Hilbert submodule and sub-Toeplitz operators}
The Hilbert sum
$$H^2_1(S)=\S_m\PL_m(Z)=\mbox{Ran}(P)$$
will be called the {\bf sub-Hardy space}. For smooth symbols $f\in\CL^\oo(S)$ define the {\bf sub-Toeplitz operator}
$$S_f:=P\,f\,P=P\,T_f\,P$$
as a bounded operator on $H^2_1(S).$ Let $\AL$ be the $*$-algebra generated by $S_p$ for polynomial symbols $p\in\PL(Z)$. For $p,q\in\PL(Z)$ we have
$$S_pS_q=S_{pq}$$
since $PT_qP^\perp=0.$ Thus it often suffices to consider linear symbols $z\mapsto(z|u)$ for some $u\in Z.$ We denote by $S_u$ the corresponding operators.
\begin{theorem}\label{r} For $\lm\in\Nl^r_+$ let $p\in\PL(Z)$ satisfy $\deg p\le|\lm'|.$ Then
$$P T_p^*T_q P=S_{T_p^*q}\qquad\forall\,q\in\PL_\lm(Z).$$
\end{theorem}
The proof is based on the following Lemma.
\begin{lemma}\label{u} Let $\lm$ be a partition and $q\in\PL_\lm(Z).$ Then we have for $u\in Z$ and each $h\in\PL_{n,0}(Z)$,
$$P_{\lm+n[1]-[i]}u^\dl P_{\lm+n[1]}h q=P_{\lm+n[1]-[i]}h P_{\lm-[i]}u^\dl q\qquad\forall\,i>1.$$
\end{lemma}
\begin{proof} Write $h q=\S_\ll P_\ll h q.$ The partitions $\ll$ occurring here satisfy $\ll\ge\lm$ and hence $\ll'\ge\lm'.$ For such $\ll$ we have
$$u^\dl P_\ll h q=\S_j P_{\ll-[j]}u^\dl P_\ll h q.$$
Now assume $\ll-[j]=\lm+n[1]-[i].$ If $j=1$ than $\ll'=\lm'-[i]\ngeq\lm'.$ Hence $j>1.$ If $j\ne i$ then $\ll'=\lm'-[i]+[j]\ngeq\lm'.$ Hence $i=j$ and therefore $\ll=\lm+n[1].$ This argument shows
\be{13}P_{\lm+n[1]-[i]}u^\dl(h q)=\S_\ll P_{\lm+n[1]-[i]}u^\dl P_\ll h q=P_{\lm+n[1]-[i]}u^\dl P_{\lm+n[1]}h q.\ee
Since $u^\dl(h q)=q(u^\dl h)+h(u^\dl q)$ and $q(u^\dl h)$ has only components $\ll\ge\lm$ which satisfy $\ll'\ge\lm'$ it follows that
\be{14}P_{\lm+n[1]-[i]}u^\dl(h q)=P_{\lm+n[1]-[i]}h(u^\dl q).\ee
We next show that
\be{15}P_{\lm+n[1]-[i]}h(u^\dl q)=\S_j P_{\lm+n[1]-[i]}h P_{\lm-[j]}u^\dl q=P_{\lm+n[1]-[i]}h P_{\lm-[i]}u^\dl q.\ee
In fact, since $h P_{\lm-[1]}u^\dl q$ cannot have a component $\ll$ with $\ll_1=\lm_1+n$ we may assume $j>1.$ If $j\ne i,$ then the components $\ll\ge\lm-[j]$ occurring in $h P_{\lm-[j]}u^\dl q$ satisfy $\ll'\ge\lm'-[j]$ which implies
$\ll'\ne\lm'-[i].$ Thus \er{15} holds. Combining equations \er{13},\er{14} and \er{15}, the assertion follows.
\end{proof}
{\bf Proof of Theorem \er{r}.} We may assume that $p(z)=(z|u_1)\cdots(z|u_k).$ Let $\ll=(\ll_1,\ll')$ be a partition such that $|\ll'|\ge k.$ Putting $[i_1,\ldots,i_k]=[i_1]+\ldots+[i_k]$ it follows from \er{16} that
$$T_{u_k}^*\ldots T_{u_1}^*\lq=\S_{i_1,\ldots,i_k}P_{\ll-[i_1,\ldots,i_k]}T_{u_k}^*\ldots P_{\ll-[i_1]}T_{u_1}^*\lq$$
for all $\lq\in\PL_\ll(Z).$ If any $i_j=1$ then $(\ll-[i_1,\ldots,i_k])'\ne 0.$ Therefore $(\ll-[i_1,\ldots,i_k])'=0$ implies that all $i_j>1.$ It follows that
\be{17}PT_{u_k}^*\ldots T_{u_1}^*\lq=P\S_{i_1>1,\ldots,i_k>1}P_{\ll-[i_1,\ldots,i_k]}T_{u_k}^*\ldots P_{\ll-[i_1]}T_{u_1}^*\lq.\ee
Moreover, if $|\ll'|>k$ we have $P T_{u_k}^*\ldots T_{u_1}^*\lq=0.$ The same argument shows
\be{18}P u_k^\dl\ldots u_1^\dl\lq=P\S_{i_1>1,\ldots,i_k>1}P_{\ll-[i_1,\ldots,i_k]}u_k^\dl\ldots P_{\ll-[i_1]}u_1^\dl\lq\ee
and $|\ll'|>k$ implies $P u_k^\dl\ldots u_1^\dl\lq=0.$ By Lemma \er{u} we have for $h\in\PL_{n,0}(Z)$
$$P_{ \lm+n[1]-[i_1]}u_1^\dl P_{\lm+n[1]}h q=P_{\lm+n[1]-[i_1]}h P_{\lm-[i_1]}u_1^\dl q.$$
Applying Lemma \er{u} to $P_{\lm-[i_1]}u_1^\dl q$, we obtain
$$P_{\lm+n[1]-[i_1,i_2]}u_2^\dl P_{\lm+n[1]-[i_1]}u_1^\dl P_{\lm+n[1]}h q=P_{\lm+n[1]-[i_1,i_2]}u_2^\dl
P_{\lm+n[1]-[i_1]}h P_{\lm-[i_1]}u_1^\dl q$$
$$=P_{\lm+n[1]-[i_1,i_2]}h P_{\lm-[i_1,i_2]}u_2^\dl P_{\lm-[i_1]}u_1^\dl q.$$
More generally,
\be{19}P_{\lm+n[1]-[i_1,\ldots,i_k]}u_k^\dl\ldots P_{\lm+n[1]-[i_1]}u_1^\dl P_{\lm+n[1]}h q
= P_{\lm+n[1]-[i_1,\ldots,i_k]}h P_{\lm-[i_1,\ldots,i_k]}u_k^\dl\ldots P_{\lm-[i_1]}u_1^\dl q.\ee
Consider
$$P h(T_{u_k}^*\ldots T_{u_1}^*q)=Ph\S_{i_1,\ldots,i_k}P_{\lm-[i_1,\ldots,i_k]}T_{u_k}^*\ldots P_{\lm-[i_1]}T_{u_1}^*q
=P\S_{i_1,\ldots,i_k}\S_\ll P_\ll h P_{\lm-[i_1,\ldots,i_k]}T_{u_k}^*\ldots P_{\lm-[i_1]}T_{u_1}^*q.$$
Note the components $\ll=(m,0)$ occurring here satisfy $\ll'=0\ge(\lm-[i_1,\ldots,i_k])'.$ Since $|\lm'|\ge k$ this implies that all
$i_j>1.$ Moreover, $m=|\ll|=n+|\lm-[i_1,\ldots,i_k]|=n+\lm_1+|\lm'-[i_1,\ldots,i_k]|=n+\lm_1.$ Therefore $\ll=(n+\lm_1,0)$ and hence
\be{7}P h(T_{u_k}^*\ldots T_{u_1}^*q)=P\S_{i_1>1,\ldots,i_k>1}P_{\lm+n[1]-[i_1,\ldots,i_k]}h P_{\lm-[i_1,\ldots,i_k]}
T_{u_k}^*\ldots P_{\lm-[i_1]}T_{u_1}^*q.\ee
We have $h q=\S_\ll P_\ll h q,$ where $\ll\ge\lm$ and the skew-partition $\ll-\lm$ is a horizontal $n$-strip. Since $\ll'\ge\lm'$ satisfies $|\ll'|\ge|\lm'|\ge k$ the condition $(\ll-[i_1,\ldots,i_k])'=0$ implies that all $i_j>1$ and in addition all terms with
$|\ll'|>k$ vanish. Assuming $|\ll'|=k$ it follows that $\ll'=\lm'$ and hence $\ll=\lm+n[1].$ This shows
$$P T_{u_k}^*\ldots T_{u_1}^*(h q)=P\S_\ll T_{u_k}^*\ldots T_{u_1}^* P_\ll h q=P T_{u_k}^*\ldots T_{u_1}^* P_{\lm+n[1]}h q.$$
With \er{17},\er{18}, \er{19} and \er{7}, we obtain
\begin{eqnarray*}P T_{u_k}^*\ldots T_{u_1}^*(h q)&=&P\S_{i_1>1,\ldots,i_k>1}P_{\lm+n[1]-[i_1,\ldots,i_k]}T_{u_k}^*\ldots P_{\lm+n[1]-[i_1]}T_{u_1}^*P_{\lm+n[1]}h q\\
&=&P\S_{i_1>1,\ldots,i_k>1}\f{P_{\lm+n[1]-[i_1,\ldots,i_k]}u_k^\dl\ldots P_{\lm+n[1]-[i_1]}u_1^\dl
 P_{\lm+n[1]}h q}{\((\lm-[i_1,\ldots,i_{k-1}])_{i_k}+\f a2(r-i_k)+b\)\ldots\(\lm_{i_1}+\f a2(r-i_1)+b\)}\\
&=&P\S_{i_1>1,\ldots,i_k>1}\f{P_{\lm+n[1]-[i_1,\ldots,i_k]}h P_{\lm-[i_1,\ldots,i_k]}u_k^\dl\ldots
P_{\lm-[i_1]}u_1^\dl q}{\((\lm-[i_1,\ldots,i_{k-1}])_{i_k}+\f a2(r-i_k)+b\)\ldots\(\lm_{i_1}+\f a2(r-i_1)+b\)}\\
&=&P\S_{i_1>1,\ldots,i_k>1}P_{\lm+n[1]-[i_1,\ldots,i_k]}h P_{\lm-[i_1,\ldots,i_k]}T_{u_k}^*\ldots
P_{\lm-[i_1]}T_{u_1}^*q=P h(T_{u_k}^*\ldots T_{u_1}^*q)\end{eqnarray*}
It follows that $PT_p^*(hq)=Ph(T_p^*q).$ Since $h\in\PL_{n,0}(Z)$ is arbitrary, $PT_p^*T_qP=PT_{T_p^*q}P=S_{T_p^*q}.$ $\quad\quad \Box$\vskip 3mm

Applying Theorem \er{r} we obtain
\begin{corollary}\label{v} If $\deg p,\deg q\le|\ll'|,$ then $PT_p^*A^\ll T_qP\in\AL.$
\end{corollary}
For $\lb\in\Nl^{r-1}_+,$ consider the projections
$$P^\lb:=\S_{m\ge\lb_1}P_{m,\lb}.$$
Then $P^0=P.$
\begin{definition} Define a diagonal operator $\lL$ on $\PL(Z)$ by
\be{9}\lL p_\ll:=\ll_1\,p_\ll,\qquad\forall\,p_\ll\in\PL_\ll.\ee
\end{definition}
Let $Q_j:=\oplus_{\ll_1=j}\PL_\ll(Z)$ be the eigenvector subspace (and denote the corresponding projection by the same notation) for $\lL$ with eigenvalue $j.$ Then we have the orthogonal decomposition $H^2(S)=\oplus_{j\in\Nl}Q_j$. We call an operator $T$ of {\bf finite propagation} if there exists a positive number $l$ such that
$$T Q_j\ic\bigoplus_{|i-j|\le l}Q_i.$$
\begin{lemma}\label{fp} Suppose the operator $T$ has the finite propagation property. If $T\lL^2$ is bounded, then $\lL^2 T$ and $T^*\lL^2$ are also bounded.
\end{lemma}
\begin{proof} By assumption, we have that $T=\oplus_{-l\le i\le l}T_i$ for some number $l$, where
$$T_i=\oplus_j Q_{j+i}TQ_j$$ is an operator of degree $i$. By grading, one sees that each $\lL^2 T_i$ is bounded iff there exists a constant $C_i$ such that $\|T_ip\|\le C_i\f{\|p\|}{j^2}$ for any index $j$ and $p\in Q_j$. Indeed, if such $C_i$ exists, then for any $p=\oplus_j p_j$,
$$\|\lL^2 T_ip\|^2=\S_j\|(i+j) T_ip_j\|^2\le C^2_i\S_j\f{(i+j)^2\|p_j\|^2}{j^2}\leq C^2_i(1+l)^2\|p\|^2.$$
Using the fact that $T\lL^2$ is bounded, for each $j$ and $p\in Q_j$, we have
$$\|T\lL^2p\|^2=j^2\|\S_{-l\le i\le l}T_ip\|^2=j^2\S_{-l\le i\le l}\|T_ip\|^2\ge j^2\|T_ip\|^2$$
for each $i$. It follows that each $\lL^2 T_i$ is bounded. Therefore $\lL^2 T=\S_{-l\le i\le l}\lL^2 T_i$ is bounded. This implies that $T^*\lL^2$ is bounded.
\end{proof}
Let $\CL$ denote the $*$-algebra generated by $T_p$ with polynomial symbol $p,$ and $\f1{\lL+t}$ together with all projections $P^\lb$, where $\lb\in\Nl^{r-1}_+$ is arbitrary. Define
$$\BL:=\{B\in\CL: B\lL^2\,\,bounded\},$$
$$\BL_\lL:=\AL(\lL+1)^{-1}+\BL=\{A(\lL+1)^{-1}+B:\,A\in\AL,B\in\BL\}.$$
It is easy to check that operators in $\CL$ have the finite propagation property. Therefore Lemma \er{fp} implies that $\BL$ and $\BL_\lL$ are invariant under taking adjoints.
\begin{lemma} $\BL$ is an ideal in $\CL.$ Moreover,
$$[\CL,(\lL+t)^{-1}]\ic\BL.$$
\end{lemma}
\begin{proof} For the first assertion it suffices to show that $BT_u\in\BL$ whenever $B\in\BL.$ Define a bounded operator $R_u$ by $R_up=P_{m+1,\lb}T_up$ for $p\in\PL_{m,\lb}(Z).$ Then
$$(\lL^2T_u-T_u\lL^2)p=\S_{i=1}^r(\lL^2-m^2)P_{(m,\lb)+[i]}T_up=((m+1)^2-m^2)P_{m+1,\lb}T_u p=(2\lL-1)R_up.$$
Therefore $BT_u\lL^2=B\lL^2T_u-B(2\lL-1)R_u$ is bounded. Thus $BT_u\in\BL.$ For the second assertion it suffices to show that $[T_u,(\lL+t)^{-1}]\in\BL.$ With the previous notation, we have
\begin{eqnarray*}[T_u,(\lL+t)^{-1}]\lL^2 p&=&m^2(T_u(\lL+t)^{-1}-(\lL+t)^{-1}T_u)p=m^2\S_{i=1}^r\(\f1{m+t}-\f1{\lL+t}\)P_{(m,\lb)+[i]}T_up\\
&=&m^2\(\f1{m+t}-\f1{m+1+t}\)P_{m+1,\lb}T_up=R_u\f{\lL^2}{(\lL+t)(\lL+1+t)}  p.\end{eqnarray*}
Therefore $[T_u,(\lL+t)^{-1}]\lL^2$ is bounded.
\end{proof}
\begin{lemma}\label{ba} $\BL_\lL$ is a (non-unital) $*$-algebra and an $\AL$-bimodule, i.e.,
$$\AL\BL_\lL+\BL_\lL\AL\ic\BL_\lL$$.
\end{lemma}
\begin{proof} We only show that $\BL_\lL\AL\ic\BL_\lL$. Indeed, for $A\in\AL, B\in\BL$, and $u\in Z$, we have
$$(A(\lL+1)^{-1}+B)S_u-AS_u(\lL+1)^{-1}=BS_u+A[(\lL+1)^{-1},S_u]=BS_u-AS_u(\lL+1)^{-1}(\lL+2)^{-1}\in\BL.$$
Since $AS_u\in\AL,$ it follows that $A(\lL+1)^{-1}+B\in\BL_\lL.$
\end{proof}
\begin{proposition}\label{q} For $\ll\in\Nl^r_+$ let $p,q\in\PL(Z)$ satisfy $\deg(p),\deg(q)\le|\ll'|.$ Then
$$PT_p^*P^{\ll'}T_q P\in\AL+\BL_\lL.$$
\end{proposition}
\begin{proof} The $K$-invariant operator $P^{\ll'}A^\ll P^{\ll'}$ is diagonal, and Proposition \er{j} implies that
$$P^{\ll'}A^\ll P^{\ll'}=\f{c_\ll\lL+\t c_\ll}{\lL+1}P^{\ll'}+B,$$
where $B\in\BL$ and $c_\ll>0.$ It follows that
$$P^{\ll'}=P^{\ll'}A^\ll P^{\ll'}\f{\lL+1}{c_\ll\lL+\t c_\ll}+B'$$
with $B'\in\BL.$ If $\lb\in\Nl^{r-1}_+$ satisfies $|\lb|\le|\ll'|$ and $A^\ll P^\lb$ is non-zero, then $\lb\ge\ll'$ by Lemma \er{z}. This is only possible if $\lb=\ll'.$ Therefore
$$PT_p^*A^\ll T_qP=\S_{|\lb|\le|\ll'|}PT_p^*P^\lb A^\ll P^\lb T_qP=PT_p^*P^{\ll'}A^\ll P^{\ll'}T_qP.$$
Since $\BL$ is an ideal in $\CL$ and $\[\f{\lL+1}{c_\ll\lL+\t c_\ll},T_qP\]\in\BL$ we obtain
\begin{eqnarray*}PT_p^*P^{\ll'}T_qP&=&PT_p^*P^{\ll'}A^\ll P^{\ll'}\f{\lL+1}{c_\ll\lL+\t c_\ll}T_qP+PT_p^*B'T_qP\\
&=&PT_p^*P^{\ll'}A^\ll P^{\ll'}T_qP\f{\lL+1}{c_\ll\lL+\t c_\ll}+PT_p^*P^{\ll'}A^\ll P^{\ll'}\[\f{\lL+1}{c_\ll\lL+\t c_\ll},T_qP\]
+PT_p^*B'T_qP\\
&=&PT_p^*A^\ll T_qP\f{\lL+1}{c_\ll\lL+\t c_\ll}+B'',\end{eqnarray*}
where $B''\in\BL.$ Since $PT_p^*A^\ll T_qP\in\AL$ by Corollary \er{v}, the assertion follows.
\end{proof}
\begin{proposition}\label{aab}  $[\AL,\AL]\ic\BL_\lL.$
\end{proposition}
\begin{proof} In view of Lemma \er{ba} it suffices to show that $[S_u^*,S_v]\in\BL_\lL.$ We may suppose that $Z$ has rank $r>1.$ By definition,
$S_v=PT_vP=T_vP-P^1T_vP.$  Note $S_v\PL_{m,0}(Z)\ic\PL_{m+1,0}(Z)$ and $\f a2(r-1)+b=\lr-1.$ Applying \er{16} it follows that
\begin{eqnarray*}(m+\lr)S_u^*S_vP_m=(m+\lr)T_u^*S_vP_m&=&u^\dl(S_vP_m)
=u^\dl(T_vP_m-P^1 T_vP_m)\\&=&(u|v)P_m+T_v u^\dl P_m-u^\dl P^1 T_vP_m\\
&=&(u|v)P_m+(m+\lr-1)T_v S_u^*P_m-u^\dl P^1 T_vP_m\\&=&(u|v)P_m+(m+\lr-1)S_vS_u^*P_m-Pu^\dl P^1 T_v P_m.\end{eqnarray*}
Thus $S_u^*S_v(\lL+\lr)=(u|v)P+S_vS_u^*(\lL+\lr-1)-Pu^\dl P^1 T_v P$ and hence
$$[S_u^*,S_v](\lL+\lr)=(u|v)P+S_vS_u^*((\lL+\lr-1)-(\lL+\lr))-Pu^\dl P^1 T_vP=(u|v)P-S_vS_u^*-Pu^\dl P^1 T_vP.$$
By \er{16} we have
$$Pu^\dl P^1T_vP=\S_m Pu^\dl P^1T_vP_m=\S_m P_mu^\dl P_{m,1}T_vP_m=(1+\f a2(r-2)+b)\S_m P_mT_u^*P_{m,1}T_vP_m$$
$$=(1+\f a2(r-2)+b)\S_m PT_u^*P^1T_vP_m=(1+\f a2(r-2)+b)PT_u^*P^1T_vP.$$
Thus Proposition \er{q} implies that $Pu^\dl P^1T_vP\in\AL+\BL_\lL,$ and the assertion follows.
\end{proof}
\begin{lemma}\label{y} $\AL\ic\left\{\S_i S_{p_i}S_{q_i}^*+B:\,p_i,q_i\in\PL(Z),\,B\in\BL_\lL\right\}.$
\end{lemma}
\begin{proof} Since the latter set contains $S_u,S_v^*,$ it suffices to show that it is invariant under multiplication by $S_u,S_v^*$. By Proposition  \er{aab} we have $[S_u^*,S_p]\in\BL_\lL$ and $[S_q^*,S_v]\in\BL_\lL.$ With Lemma \er{ba}, the assertion follows.
\end{proof}
The following technical lemma will be used in the next section.
\begin{lemma}\label{ap} Let $T\in\AL+\BL_\lL.$ Then $\left\{\f{(N_1^m|TN_1^m)_S}{\|N_1^m\|_S^2}\right\}_m\in\SL.$
\end{lemma}
\begin{proof} By Lemma \er{i}, we have
$$S_{N_1^k}^*N_1^m=\f{(m+1)_k^*}{(m+\lr)_k^*}N_1^{m-k}$$
for $0\le k\le m,$ $S_{N_1^k}^*N_1^m=0$ for $k>m$ and $S_v^*N_1^m=0$ for all $v\in Z_1^\perp.$ Thus for any $p,q\in\PL(Z)$ there exist constants $c_k(p,q),$ for $0\le k\le M(p,q):=min(\deg p,\deg q),$ such that
$$(S_p^*N_1^m|S_q^*N_1^m)_S=\S_{k=0}^{M(p,q)}c_k(p,q)\|S_{N_1^k}^*N_1^m\|_S^2=\S_{k=0}^{M(p,q)}c_k(p,q)\|\f{(m+1)_k^*}{(m+\lr)_k^*}N_1^{m-k}\|_S^2$$
for all $m\ge M(p,q).$ Since $T\in\AL+\BL_\lL$, Lemma \er{y} implies that $$T=\S_i S_{p_i}S_{q_i}^*+B_0 $$
 for some  polynomials
$p_i,q_i$ and $B_0=A_1(\lL+1)^{-1}+B_1\in\BL_\lL$ with $A_1\in\AL, B_1\in\BL $. Using  Lemma \er{y} for $A_1$ again, there exist polynomials
$\lf_j, \lq_j$ and $ B_2\in\BL_\lL$ such that
$$T=\S_i S_{p_i}S_{q_i}^*+\(\S_j S_{\lf_j}S_{\lq_j}^*+B_2\)(\lL+1)^{-1}+B_1
=\S_i S_{p_i}S_{q_i}^*+\S_j S_{\lf_j}S_{\lq_j}^*(\lL+1)^{-1}+B,$$
where $B\in\BL.$ It follows that
$$(N_1^m|TN_1^m)_S-(N_1^m|BN_1^m)_S=\S_i(S_{p_i}^*N_1^m|S_{q_i}^*N_1^m)_S+\f1{m+1}\S_j(S_{\lf_j}^*N_1^m|S_{\lq_j}^*N_1^m)_S$$
$$=\S_i\S_{k=0}^{M(p_i,q_i)}c_k(p_i,q_i)\,\|\f{(m+1)_k^*}{(m+\lr)_k^*}N_1^{m-k}\|_S^2+\f1{m+1}\S_j\S_{k=0}^{M(\lf_j,\lq_j)}c_k(\lf_j,\lq_j)\,\|\f{(m+1)_k^*}{(m+\lr)_k^*}N_1^{m-k}\|_S^2.$$
Since
$$BN_1^m=(B(\lL+1)^2)(\lL+1)^{-2}N_1^m=\f{B(\lL+1)^2 N_1^m}{(m+1)^2}$$
the sequence $\{m^2\f{(N_1^m|BN_1^m)_S}{\|N_1^m\|_S^2}\}$ is bounded. Thus there exist finitely many sequences $\{c_k(m)\}\in\SL$  such that
$$\f{(N_1^m|TN_1^m)_S}{\|N_1^m\|_S^2}=\S_{k} c_k(m)\f{\|N_1^{m-k}\|_S^2}{\|N_1^m\|_S^2}.$$
This yields the desired result since $\left\{\f{\|N_1^{m-k}\|_S^2}{\|N_1^m\|_S^2}\right\}_m\in\SL_+$ by Lemma \er{qn}.
\end{proof}

\section{First main theorem}
\begin{theorem}\label{mr} Let $f\in\PL(Z\xx\o Z)$ be a real-analytic polynomial. Then $S_f\in\AL+\BL_\lL.$
\end{theorem}
The proof is based on a lengthy induction argument. We may assume that $f=\o p\,q$ for some $p,q\in\PL(Z).$
Let $\AL_{i,j}$ denote the set of all operators $PT_p^*T_qP,$ where $\deg p\le i,\,\deg q\le j.$ For a given $k$ we consider the following assumption
\be{20}\AL_{i,j}\ic\AL+\BL_\lL\mbox{ whenever }min(i,j)<k.\ee
We now proceed via a sequence of 'claims' which are proved under this assumption.
\begin{claim}\label{p} The assumption \er{20} implies that for each partition $\ll$ with $|\ll|<k$ there exist constants $a^\ll_\lb,b^\ll_\lb$ such that
\be{23}A^\ll-\S_{\lb\le\ll}\f{a^\ll_\lb\lL+b^\ll_\lb}{\lL+1}P^\lb\in\BL.\ee
\end{claim}
\begin{proof} For $\lb\le\ll\le(m,\lb)$ we have $N_{m,\lb}=N_{\lb_1,\lb}N_1^{m-\lb_1}$ and hence
$$\f{(N_{m,\lb}|A^\ll N_{m,\lb})_S}{\|N_{m,\lb}\|_S^2}=\S_i\f{\|PT_{p_i}^*N_{m,\lb}\|_S^2}{\|N_{m,\lb}\|_S^2}=\S_i\f{\|PT_{p_i}^*T_{N_{\lb_1,\lb}}N_1^{m-\lb_1}\|_S^2}{\|N_1^{m-\lb_1}\|_S^2}\f{\|N_1^{m-\lb_1}\|_S^2}{\|N_{m,\lb}\|_S^2}.$$
Since $\deg(p_i)=|\ll|<k,$ \er{20} implies $PT_{p_i}^*T_{N_{\lb_1,\lb}}P\in\AL+\BL_\lL.$ By Lemma \er{ap} and Lemma \er{qn}, we have that
$\left\{\f{(N_{m,\lb}|A^\ll N_{m,\lb})_S}{\|N_{m,\lb}\|_S^2}\right\}_m\in\SL$. By \er{32} there is a sequence $\oL_m,$ with $m^2\,\oL_m$ bounded, such that
$$A^\ll=\S_{\lb\le\ll\le(m,\lb)}\f{(N_{m,\lb}|A^\ll N_{m,\lb})_S}{\|N_{m,\lb}\|_S^2}\,P_{m,\lb}
=\S_{\lb\le\ll\le(m,\lb)}\(\f{a^\ll_\lb m+b^\ll_\lb}{m+1}+\oL_{m}\)P_{m,\lb}=\S_{\lb\le\ll}\f{a^\ll_\lb\lL+b^\ll_\lb}{\lL+1}P^\lb+B,$$
where we set $a^\ll_\lb=b^\ll_\lb=0$ if $\lb\ngeq\ll'.$ Thus $B-\S_{\lb\le\ll\le(m,\lb)}\oL_m\,P_{m,\lb}$ has finite rank and hence $B\in\BL$.
\end{proof}
\begin{claim}\label{m} Under the assumption \er{20} there exist constants $c^\lb_\la,d^\lb_\la$ such that
\be{24}P^\lb-\S_{\la\le\lb}\f{c^\lb_\la\lL+d^\lb_\la}{\lL+1}A^\la\in\BL,\qquad\forall\,\lb\in\Nl_+^{r-1},\,|\lb|<k.\ee
\end{claim}
\begin{proof} We use induction on $|\lb|.$ The case $\lb=0$ is trivial. Assume
\er{24} holds for all $\lb$ with $|\lb|<j<k$. Let $\lb$ satisfy $|\lb|=j.$ Then Claim \er{p} implies
\be{251}A^\lb=\f{a^\lb_\lb\lL+b^\lb_\lb}{\lL+1}P^\lb+\S_{\la<\lb}\f{a^\lb_\la\lL+b^\lb_\la}{\lL+1}P^\la+B^\lb,\ee
where $B^\lb\in\BL$, and $\la<\lb$ means that $\la\leq\lb$ and $\la\neq \lb$. Now consider the diagonal operator
$$\S_{|\ll|=j}A^\ll=\S_{\lm\in\Nl^r_+}a_\lm P_\lm.$$
If $|\ll|=j,$ then $(N_{m,\lb}|A^\ll N_{m,\lb})_S$ is non-zero only if $\ll=\lb,$ since $\lb\le\ll$ and $|\lb|=j=|\ll|.$ Therefore
$$a_{m,\lb}\,N_{m,\lb}=\S_{|\ll|=j}A^\ll N_{m,\lb}=\S_{|\ll|=j}\f{(N_{m,\lb}|A^\ll N_{m,\lb})_S}{\|N_{m,\lb}\|_S^2}\,N_{m,\lb}
=\f{(N_{m,\lb}|A^\lb N_{m,\lb})_S}{\|N_{m,\lb}\|_S^2}\,N_{m,\lb}.$$
By \cite[Theorem 1.6]{U2}, there exists a constant $c>0$ such that $a_\lm\ge c$ whenever $\lm_2+\cdots+\lm_r=j.$ For
$\lm=(m,\lb)$ this implies $\f{(N_{m,\lb}|A^\lb N_{m,\lb})_S}{\|N_{m,\lb}\|_S^2}=a_{m,\lb}\ge c$ and hence $a^\lb_\lb
=\lim_m \f{(N_{m,\lb}|A^\lb N_{m,\lb})_S}{\|N_{m,\lb}\|_S^2}\ge c>0.$
For any $|\la|<|\lb|=j,$ the induction hypothesis implies
$$P^\la=\S_{\lg\le\la}\f{c^\la_\lg\lL+d^\la_\lg}{\lL+1}A^\lg+B^\la,$$
where $B^\la\in\BL.$ Plugging into \er{251} we obtain
$$P^\lb=\f{\lL+1}{a_\lb^\lb\lL+b_\lb^\lb}\[A^\lb-\S_{\la<\lb}\f{a_\la^\lb\lL+b_\la^\lb}{a_\lb^\lb\lL+b_\lb^\lb}\(\S_{\lg\le\la}
\f{c_\lg^\la\lL+d_\lg^\la}{\lL+1}A^\lg+B^\la\)-B^\lb\].$$
It is easy to see that this expression has the desired form.
\end{proof}
\begin{claim}\label{t} The assumption \er{20} implies $\AL_{k,k}\ic\AL+\BL_\lL.$
\end{claim}
\begin{proof} Let $\deg p=\deg q=k.$ Then
$$PT_p^*T_qP=\S_{|\lb|\le k}PT_p^*P^\lb T_qP.$$
If $|\lb|=k$ then $PT_p^*P^\lb T_qP\in\AL+\BL_\lL$ by Proposition \er{q}. If $|\lb|=h<k$ and $\la\le\lb$ then \er{20} implies
$PT_p^*A^\la T_qP\in\AL_{k,h}\AL_{h,k}\ic\AL+\BL_\lL.$ It follows that
$$PT_p^*A^\la\f{c^\lb_\la\lL+d^\lb_\la}{\lL+1}T_qP=PT_p^*A^\la T_qP\f{c^\lb_\la\lL+d^\lb_\la}{\lL+1}
+PT_p^*A^\la\[\f{c^\lb_\la\lL+d^\lb_\la}{\lL+1},T_qP\]\in\AL+\BL_\lL,$$
since $\[\f{c^\lb_\la\lL+d^\lb_\la}{\lL+1},\CL\]\ic\BL$ and $\BL$ is an ideal in $\CL$. Therefore Claim \er{m} implies $PT_p^*P^\lb T_qP\in\AL+\BL_\lL.$
\end{proof}
\begin{claim}\label{w} Under the assumption \er{20}, for $T\in\CL$ and $q\in\PL(Z)$ of degree $i<k$ there exists $B\in\BL$ such that
$$PT[T_u^*,T_v]T_qP=B+\S_{|\lb|\le i}\S_{\la\le\lb}\S_{\lg\le\lb}PT A^\la[T_u^*,T_v]A^\lg T_qP
\f{c^\lb_\la c^\lb_\lg\lL+c_\la^\lb(d^\lb_\lg-c_\lg^\lb)+(d^\lb_\la-c_\la^\lb)c_\lg^\lb}{\lL+1}.$$
\end{claim}
\begin{proof} Since $\mbox{Ran}(T_qP)\ic\S_{|\lb|\le i}P^\lb$ and $[T_u^*,T_v]$ is a 'block-diagonal' operator \cite[Lemma 2.1]{U2} which commutes with each $P^\lb,$ it suffices to consider $PTP^\lb[T_u^*,T_v]P^\lb T_qP$ for $\lb\in\Nl^{r-1}_+$ satisfying $|\lb|\le i.$ By Claim \er{m} we have
$$PTP^\lb[T_u^*,T_v]P^\lb T_qP=PT\(B_1+\S_{\la\le\lb}\f{c^\lb_\la\lL+d^\lb_\la}{\lL+1}A^\la\)[T_u^*,T_v]
\(B_2+\S_{\lg\le\lb}\f{c^\lb_\lg\lL+d^\lb_\lg}{\lL+1}A^\lg\)T_qP,$$
where $B_1,B_2\in\BL.$ Since $\BL\ic\CL$ is an ideal and $\CL$ contains $PT,\,[T_u^*,T_v],\,A^\la,\,A^\lg,\,T_qP,$ the assertion follows.
\end{proof}
\begin{claim}\label{s} The assumption \er{20} implies
\be{26}PT_\lf^*[T_u^*,T_v]T_\lq P\in\AL+\BL_\lL\ee
whenever $\deg\lf,\deg\lq<k.$
\end{claim}
\begin{proof} We prove \er{26} by induction on $h=\max(\deg\lf,\deg\lq)<k$. For $h=0$, we have
$$P[T_u^*,T_v]P=PT_u^*T_vP-PT_vPT_u^*P,$$
where $PT_u^*T_vP\in\AL_{1,1}\ic\AL+\BL_\lL$ by Claim \er{t}, and $PT_vPT_u^*P\in\AL.$
For the induction step, let $\lf,\lq$ be polynomials with $\deg\lf\le h=\deg\lq<k,$ and we may assume that \er{26} holds in the case of the maximal degree less than $h.$ Then
$$PT_\lf^*[T_u^*,T_v]T_\lq P=PT_\lf^*\([T_u^*,T_{v\lq}]-T_v[T_u^*,T_\lq]\)P
=PT_{\lf u}^*T_{v\lq}P-PT_\lf^*T_{v\lq}PT_u^*P-PT_\lf^*T_v[T_u^*,T_\lq]P.$$
By the assumption \er{20}, we have $PT_\lf^*T_{v\lq}P\in\AL_{h,h+1}\ic\AL+\BL_\lL,$ and using Claim \er{t} also $PT_{\lf u}^*T_{v\lq}P\in\AL_{h+1,h+1}\ic\AL+\BL_\lL.$ For the third term we may assume that $\lq=v_{h-1}\,\cdots v_0$ for some linear functions $v_i.$ Then
$$PT_\lf^*T_v[T_u^*,T_\lq]P=\S_{i=0}^{h-1}PT_\lf^*T_{v\cdot v_{h-1}\cdots v_{i+1}}[T_u^*,T_{v_i}]T_{v_{i-1}\cdots v_0}P.$$
If $p,q,\lx,\lh$ are polynomials of degree $\le i<h$ we have
$$PT_\lf^*T_{v\cdot v_{h-1}\cdots v_{i+1}}T_pPT_q^*[T_u^*,T_{v_i}]T_\lx PT_\lh^*T_{v_{i-1}\cdots v_0}P\in\AL+\BL_\lL,$$
since \er{20} implies that $\AL+\BL_\lL$ contains $PT_\lf^*T_{v\cdot v_{h-1}\cdots v_{i+1}}T_pP\in\AL_{h,h}$ and
$PT_\lh^*T_{v_{i-1}\cdots v_0}P\in\AL_{i,i},$ and the induction hypothesis implies $PT_q^*[T_u^*,T_{v_i}]T_\lx P\in\AL+\BL_\lL.$ Thus
$$PT_\lf^*T_{v\cdot v_{h-1}\cdots v_{i+1}}A^\la[T_u^*,T_{v_i}]A^\lr T_{v_{i-1}\cdots v_0}P\ic\AL+\BL_\lL,$$
whenever $|\la|\le i$ and $|\lr|\le i.$ Now the assertion follows from Claim \er{w}
\end{proof}
The {\bf proof of Theorem \er{mr}} can now be completed as follows. Since $\AL_{m,n}^*=\AL_{n,m},$ it suffices to show that
\be{28}\AL_k:=\S_{\el\ge k}\AL_{\el,k}\ic\AL+\BL_\lL.\ee
We prove \er{28} by induction over $k\ge 0.$ The case $k=0$ is trivial. For the induction step, let $k>0$ and suppose that
$\AL_h\ic\AL+\BL_\lL$ whenever $h<k.$ This is precisely the assumption \er{20}. We prove that
\be{29}\AL_{\el,k}\ic\AL+\BL_\lL\ee
by induction over $\el\ge k.$ By Claim \er{t} we have $\AL_{k,k}\ic\AL+\BL_\lL.$ For the induction step assume that
$\AL_{\el,k}\ic\AL+\BL_\lL$ for some $\el\ge k$. Passing to $\el+1$, consider polynomials $\lf,\lq$ with $\deg\lf\le\el$ and $\deg\lq=k.$ Then we have for any linear function $u$
$$PT_{\lf\cdot u}^*T_\lq P=PT^*_\lf T_u^*T_\lq P=PT^*_\lf T_\lq PT_u^*P+PT^*_\lf[T_u^*,T_\lq]P.$$
By the induction hypothesis we have $PT_\lf^*T_\lq P\in\AL_{\el,k}\ic\AL+\BL_\lL.$ For the second term, we may assume that
$\lq=v_{k-1}\cdots v_0$ for some linear functions $v_i.$ Then
$$PT_\lf^*[T_u^*,T_\lq]P=\S_{i=0}^{k-1}PT_\lf^*T_{v_{k-1}\cdots v_{i+1}}[T_u^*,T_{v_i}]T_{v_{i-1}\cdots v_0}P.$$
If $p,q,\lx,\lh$ are polynomials of degree $\le i<k$ we have
$$PT_\lf^*T_{v_{k-1}\cdots v_{i+1}}T_pPT_q^*[T_u^*,T_{v_i}]T_\lx PT_\lh^*T_{v_{i-1}\cdots v_0}P\in\AL+\BL_\lL,$$
since the assumption \er{20} implies that $\AL+\BL_\lL$ contains $PT_\lf^*T_{v_{k-1}\cdots v_{i+1}}T_pP\in\AL_{\el,k-1}$ and
$PT_\lh^*T_{v_{i-1}\cdots v_0}P\in\AL_{i,i},$ and Claim \er{s} implies $PT_q^*[T_u^*,T_{v_i}]T_\lx P\in\AL+\BL_\lL.$ Thus
$$PT_\lf^*T_{v_{k-1}\cdots v_{i+1}}A^\la[T_u^*,T_{v_i}]A^\lg T_{v_{i-1}\cdots v_0}P\in\AL+\BL_\lL,$$
whenever $|\la|\le i$ and $|\lg|\le i.$ With Claim \er{w}, it follows that $PT_\lf^*[T_u^*,T_\lq]P\in\AL+\BL_\lL.$ Therefore $\AL_{\el+1,k}\ic\AL+\BL_\lL,$ completing the induction proof of \er{28}.

\section{Smooth extension and Dixmier trace}
Let $\KL$ denote the compact operators. By definition \cite{C2} we have
$$\LL^{n,\oo}:=\{T\in\KL:\,\lm_j(T)=O(j^{-1/n})\}$$
for $n>1,$ and
$$\LL^{1,\oo}:=\{T\in\KL:\,\S_{i=1}^j\lm_i(T)=O(\log j)\}.$$
Here $\lm_1(T)\ge\lm_2(T)\ge\cdots$ are the singular values of $T.$ We will apply these concepts to the Hilbert space $H^2_1(S).$
Using the invariants $a,b$ we put
$$n:=1+a(r-1)+b.$$
Note that $n$ is not the dimension $d=r(1+\f a2(r-1)+b)$ of the underlying domain $D,$ unless $r=1.$ We will give a geometric interpretation below.
\begin{lemma}\label{c} Consider $\lL$ as an unbounded operator on $H^2_1(S)$. Then $(\lL+1)^{-1}\in\LL^{n,\oo}.$
\end{lemma}
\begin{proof} For any partition $\ll$ it follows from \cite[Lemma 2.7 and Lemma 2.6]{U1} that
\be{10}\dim\PL_\ll(Z)=\f{(\lr)_\ll}{(\lr-b)_\ll}\P_{1\le i<j\le r}\f{\ll_i-\ll_j+\f a2(j-i)}{\f a2(j-i)}\cdot\f{(\ll_i-\ll_j+1+\f a2(j-i-1))_{a-1}}{(1+\f a2(j-i-1))_{a-1}}.\ee
Specializing \er{10} to $m=(m,0,\ldots,0)$ we obtain
$$\dim\PL_m(Z)=\f{(m+1+\f a2(r-1))_b}{(1+\f a2(r-1))_b}\P_{j=2}^r\f{m+\f a2(j-1)}{\f a2(j-1)}\cdot\f{(m+1+\f a2(j-2))_{a-1}}{(1+\f a2(j-2))_{a-1}}$$
for $m\ge b.$ It follows that asymptotically, we have
$$\dim\PL_m(Z)\sim c\cdot m^{b+a(r-1)}=c\cdot m^{n-1}$$
for some constant $c>0$ independent of $m.$ Since $(\lL+1)^{-1}$ has the eigenvalues $1/(1+m),$ with eigenspace $\PL_m(Z),$ this estimate implies that the partial sum
$$S_j((\lL+1)^{-1})=\S_{i=0}^{j}\lm_i((\lL+1)^{-1})\sim j^{1-1/n},$$
where $\lm_i(T)$ is the $i$-th eigenvalue of $T.$ This implies the assertion since, for $n>1,$ $T\in\LL^{n,\oo}$ iff $\{j^{(1/n-1)}S_j(T):\,j\ge 1\}$ is a bounded sequence \cite{C2}.
\end{proof}
\begin{theorem}\label{a} Let $f,g\in\PL(Z\xx\o Z)$ be real-analytic polynomials. Then $[S_f,S_g]\in\LL^{n,\oo}.$
\end{theorem}
\begin{proof} Let  $A\in\AL,B\in\BL.$ Then Lemma \er{c} implies $A(\lL+1)^{-1}+B \in\LL^{n,\oo}$ since $A+B(\lL+1)$ is bounded. It follows that $\BL_\lL\ic\LL^{n,\oo}.$ Since $[\AL,\AL]\ic\BL_\lL$ by Proposition \er{aab} and $\BL_\lL$ is an $\AL$-bimodule we obtain
$$[\AL+\BL_\lL,\AL+\BL_\lL]\ic\BL_\lL.$$
Since $S_f,S_g\in\AL+\BL_\lL$ by Theorem \er{mr}, the assertion follows.
\end{proof}
It is well known \cite{C2} that $T_i\in\LL^{p_i,\oo}$ and $\S_{i=1}^n\f1{p_i}=1$ implies $T=T_1\cdots T_n\in\LL^{1,\oo}.$ Hence Theorem \er{a} implies
\begin{corollary} Let $f_1,g_1,\ldots f_n,g_n\in\PL(Z\xx\o Z)$ be real-analytic polynomials. Then
\be{33}[S_{f_1},S_{g_1}]\cdots[S_{f_n},S_{g_n}]\in\LL^{1,\oo}.\ee
\end{corollary}
The trace class $\LL^1$ is a proper subspace of $\LL^{1,\oo}$. For $T\in\LL^{1,\oo}$ the {\bf Dixmier trace}, denoted by
$tr_\lo(T)$, depends a priori on a choice of positive functional $\lo$ on $l^\oo(\Nl)$ vanishing on $c_0(\Nl)$. For the so-called {\bf measurable} operators $T$ the value $tr_\lo(T)$ is independent of $\lo$. More precisely, for a positive operator $T,$
$$tr_\lo(T) =\lim_{j\to\oo}\f1{\log j}{\S_{i=1}^j\lm_i(T)}$$
whenever the limit exists. It also satisfies the tracial property
$$tr_\lo(TS)=tr_\lo(ST)$$
and $tr_\lo(T)=0$ if $T\in\LL^1.$ We refer the reader to \cite{C2} for more details.

In order to determine the Dixmier trace of the operators \er{33} we consider the algebraic variety
$$Z_1^\ol:=\{z\in Z:\,rank(z)\le 1\},$$
which has (complex) dimension $\dim Z_1^\ol=1+a(r-1)+b=n,$ and is singular only at the origin.
\begin{proposition} Consider the polynomial ideal $\IL(Z_1^\ol)\ic\PL(Z)$ vanishing on $Z_1^\ol.$ Then the sub-Hardy space $H^2_1(S)$ can be identified with the Hilbert quotient module
$$H^2_1(S)=H^2(S)/\o{\IL(Z_1^\ol)}\al\o{\IL(Z_1^\ol)}^\perp.$$
\end{proposition}
\begin{proof} It suffices to show that $\IL(Z_1^\ol)$ coincides with the ideal
$$\JL=\bigoplus_{\ll_2>0}\PL_\ll(Z)\ic\PL(Z).$$
For $z\in Z_1^\ol$, we have $N_\el(z)=0$ for $\el\ge 2$ since $rank(z)\le 1.$ This implies that $\JL\ic\IL(Z_1^\ol).$ By Schur orthogonality the orthogonal projection $P_\ll$ is given by
$$\f1{d_\ll}P_\ll\,f=\S_\la\I_{K}dk\,(\lf_\la|k\cdot\lf_\la)\,(k^{-1}\cdot f)$$
for all $f\in\PL(Z),$ where $\lf_\la\in\PL_\ll(Z)$ is an orthonormal basis. It follows that the $K$-invariant ideal $\IL(Z_1^\ol)$
is invariant under all $P_\ll.$ Now suppose there exists $f\in\IL(Z_1^\ol)\sm\JL.$ Then $f=f'+f'',$ where $f''\in\JL$ and $f'\in\JL^\perp=\bigoplus_{m}\PL_m(Z)$ is non-zero. Since $\JL\ic\IL(Z_1^\ol)$ we may assume $f=f'.$ By the above, we may assume that
$f\in\PL_m(Z)$ for some $m\ge 0.$ By irreducibiliy, it follows that $\PL_m(Z)\ic\IL(Z_1^\ol),$ which is a contradiction since
$N_1^m\notin\IL(Z_1^\ol).$
\end{proof}
The unit ball $D\ui Z_1^\ol$ of $Z_1^\ol$ is a {\bf strictly pseudo-convex domain} (singular at the origin), with a $K$-homogeneous smooth boundary $S_1=\{c:\,\{cc^*c\}=c,\,rank(c)=1\}$ consisting of all minimal tripotents. Denote by $L^2(S_1)$ the $L^2$-space with respect to the $K$-invariant measure. The Hardy space $H^2(S_1)$ is the closure of the algebra $\PL(Z)$ of all polynomials on $Z,$ restricted to $S_1.$ Since $N_\el|_{S_1}=0$ for each $\el\ge 2$, it follows that
$$H^2(S_1)=\S_{m\ge 0}\PL_m^\sim(Z),$$
where $\t f=f|_{S_1}$ denotes the restriction.
\begin{lemma} Let $p,q\in\PL_m(Z).$ Then
$$(p|q)_S=\f{(ra/2)_m}{(a/2)_m}(\t p|\t q)_{S_1}.$$
Hence the transformation $U:H^2_1(S)\to H^2(S_1),$ defined by
$$Up:=\F{\f{(ra/2)_m}{(a/2)_m}}\t p\qquad\forall\,p\in\PL_m(Z),$$
is unitary.
\end{lemma}
\begin{proof} Let $X$ be the self-adjoint part of the Peirce $2$-space $Z^2_e$ of full rank $r$ \cite{FK}. For any partition $\ll\in\Nl^r_+$, the associated {\bf spherical polynomial} $\lf^\ll$ on $X\ic Z,$ normalized by $\lf^\ll(e)=1$ \cite{FK}, is given by
$$\f{\EL^\ll(t,e)}{d_\ll}=\f{\lf^\ll(t)}{(d/r)_\ll}$$
for all $t\in X$, where $d_\ll:=\dim\,\PL_\ll(Z)$ and $\EL^\ll(z,w)$ is the Fischer-Fock reproducing kernel for $\PL_\ll(Z).$ Now suppose $\ll\in\Nl^\el_+$. Then \cite[Proposition 3.7]{AU} implies
$$\lf^\ll(e_1+\ldots+e_\el)=\f{(\el a/2)_\ll}{(r a/2)_\ll}$$
and hence
$$\lf^\ll(t)=\f{(\el a/2)_\ll}{(r a/2)_\ll}\lf^\ll_\el(t)$$
for all $t\in X_\el\ic X$, where $\lf^\ll_\el$ is the spherical polynomial for the self-adjoint part $X_\el$ of the Peirce $2$-space
$Z^2_{e_1+\ldots e_\el}.$ Let $\lO_\el\ic X_\el$ be the strictly positive cone, and let $t\in\lO_\el$ be fixed. By Schur orthogonality \cite[Theorem 14.3.3]{D}, we have
$$\I_K dk\,\EL^\ll(z,\,k\F{t})\EL^\lm(k\F{t},w)=\f{\ld_{\ll,\lm}}{d_\ll}\EL^\ll(k\F{t},k\F{t})\,\EL^\ll(z,w)$$
$$=\f{\ld_{\ll,\lm}}{(d/r)_\ll}\lf^\ll(t)\,\EL^\ll(z,w)=\f{\ld_{\ll,\lm}}{(d/r)_\ll}\f{(\el a/2)_\ll}{(r a/2)_\ll}
\lf^\ll_\el(t)\EL^\ll(z,w).$$
for all $\ll,\lm\in\Nl^\el_+$ and $z,w\in Z.$ Applying this identity to $\el=r,t=e$ and $\el=1,t=e_1,$ resp., the assertion follows.
\end{proof}
Define $\t\lL\t p=m\t p$  for  $p\in\PL_m(Z)$. Then Lemma \er{c} gives
$$\f1{1+\t\lL}\in\LL^{n,\oo}.$$
Let $\t T_f$ denote the Toeplitz operators $H^2(S_1).$ Then $\t T_u\t\lL=(\t\lL-1)\t T_u$ and $\t T_u^*\t\lL=(\t\lL+1)\t T_u^*.$
\begin{proposition}\label{ss} Let $u,v\in Z.$ Then $U S_u U^*-\t T_u\in\LL^{n,\oo}$ and
$$U[S_u,S_v^*]U^*-[\t T_u,\t T_v^*]\in \LL^{n/2,\oo}.$$
\end{proposition}
\begin{proof} For each $p\in\PL_m(Z)$ we have
\begin{eqnarray*}US_u U^*\t p&=&U^*PT_u\(\F{\f{(a/2)_m}{(ra/2)_m}}p\)=\F{\f{(a/2)_m}{(ra/2)_m}}U^*P(up)\\
&=&\F{\f{(a/2)_m}{(ra/2)_m}}\F{\f{(ra/2)_{m+1}}{(a/2)_{m+1}}}\w{P(up)}=\F{\f{ra/2+m}{a/2+m}}\t u\t p=\t T_u\F{\f{\t \lL+ra/2}{\t \lL+a/2}}\t p.\end{eqnarray*}
Thus we have $U S_u U^*=\t T_u\F{\f{\t\lL+ra/2}{\t\lL+a/2}}.$ This implies the first assertion. For the second assertion
\begin{eqnarray*}U[S_u,S_v^*]U^*&=&\[\t T_u\F{\f{\t\lL+ra/2}{\t\lL+a/2}},\F{\f{\t\lL+ra/2}{\t\lL+a/2}}\t T_v^*\]
=\t T_u{\f{\t\lL+ra/2}{\t\lL+a/2}}\t T_v^*-\F{\f{\t\lL+ra/2}{\t\lL+a/2}}\t T_v^*\t T_u\F{\f{\t\lL+ra/2}{\t\lL+a/2}}\\
&=&{\f{\t\lL+ra/2-1}{\t\lL+a/2-1}}\t T_u\t T_v^*-{\f{\t\lL+ra/2}{\t\lL+a/2}}\t T_v^* \t T_u
={\f{\t\lL+ra/2}{\t\lL+a/2}}[\t T_u,\t T_v^*]+ \[{\f{\t\lL+ra/2-1}{\t\lL+a/2-1}}-{\f{\t\lL+ra/2}{\t\lL+a/2}}\]\t T_u\t T_v^*.\end{eqnarray*}
Therefore
\begin{eqnarray*}U[S_u,S_v^*]U^*-[\t T_u,\t T_v^*]&=&\(\f{\t\lL+ra/2}{\t\lL+a/2}-1\)[\t T_u,\t T_v^*]
+\({\f{\t\lL+ra/2-1}{\t\lL+a/2-1}}-{\f{\t\lL+ra/2}{\t\lL+a/2}}\)\t T_u\t T_v^*\\
&=&\f{(r-1)a/2}{\t\lL+a/2}[\t T_u,\t T_v^*]+{\f{(r-1)a/2}{(\t\lL+a/2-1)(\t\lL+a/2)}}\t T_u\t T_v^*\in\LL^{n/2,\oo}\end{eqnarray*}
since $[\t T_u,\t T_v^*]$ (cf. \cite{EZ}) and $(\t\lL+a/2)^{-1}$ belong to $\LL^{n,\oo}.$
\end{proof}
To consider general symbols, we need the following algebraic lemma.
\begin{lemma}\label{sa} Suppose that the given operators $A_i,\t A_j$ satisfy that $A_i-\t A_i,\,[A_i, A_j],\,[\t A_i,\t A_j]\in\LL^{n,\oo}$ and $[A_i, A_j]-[\t A_i,\t A_j]\in\LL^{n/2,\oo}$ for $1\leq i,j\leq 4.$ Then
$$[A_1A_2,A_3A_4]-[\t A_1\t A_2,\t A_3\t A_4]\in\LL^{n/2,\oo}.$$
\end{lemma}
\begin{corollary}
For polynomials $p,q,\lf,\lq$, we have
$$U[S^*_p S_q,S^*_\lf S_\lq]U^*-[\t T_{\o p q},\t T_{\o\lf\lq}]\in\LL^{n/2,\oo}.$$
\end{corollary}
\begin{proof} Apply Lemma \er{sa} and Proposition \er{ss}.
\end{proof}
Every $f\in\CL^\oo(S)$ has a {\bf Poisson integral extension} $\h f\in\CL^\oo(D),$ which is harmonic in the sense that it is annihilated by the so-called Hua operators \cite{K,S1}. For any non-zero tripotent $c\in S_k$ there exists a continuous extension, again denoted by $\h f,$ onto the boundary component $c+D_c.$ This extension is given by
$$\h f(c+\lz)=f_c^\yi(\lz)$$
for all $\lz\in D_c^0,$ where $f_c^\yi$ denotes the Poisson extension, relative to the Shilov boundary $S_c$ of $D_c,$ for the restricted smooth function
$$f_c(\lz):=f(c+\lz),\qquad \lz\in S_c.$$
Setting $\lz=0$ the Poisson extension $\h f$ is well-defined on $S_k.$
\begin{lemma}\label{zz} For all $c\in S_k$ we have
$$\widehat{\o p q}(c)=(p_c|q_c)_{S_c}.$$
\end{lemma}
\begin{proof} Let $h(z)=\widehat{\o p q}(z)$ be the Poisson extension of $\o{p(s)}q(s)$. For all $\lz\in S_c$, we have $c+\lz\in S$ and hence
$$h_c(\lz)=h(c+\lz)=\o{p(c+\lz)}q(c+\lz)=\o{p_c(\lz)}q_c(\lz)$$
Since $h_c(\lz)$ is harmonic, the mean value property applied to the Peirce $0$-space $Z_c$ yields
$$h(c)=h_c(0)=h_c^\yi(0)=\I_{S_c}d\lz\,h_c(\lz)=\I_{S_c}d\lz\,\o{p_c(\lz)}q_c(\lz)=(p_c|q_c)_{S_c}.$$
\end{proof}
\begin{proposition}\label{df} For polynomials $f\in\PL(Z\xx\o Z)$ we have $U S_f U^*-\t T_{\h f}\in\LL^{n,\oo}$ and
$$U [S^*_f, S_f] U^*-[\t T^*_{\h f},\t T_{\h f}]\in\LL^{n/2,\oo}.$$
Here $\h f$ is the Poisson extension restricted to $S_1$.
\end{proposition}
\begin{proof} Without loss of generality we may suppose that $f=\o p q$ for $p,q\in\PL(Z).$  By Theorem \er{mr}, Proposition \er{aab}  and Lemma \er{y}, there exist $B\in\BL_\lL$ and finitely many  $p_i,q_i\in\PL(Z)$ such that
$$S_f=PT_p^*T_q P=B+\S_i S_{p_i}^*S_{q_i}.$$
By the definition of $\BL_\lL$, Proposition \er{aab} and Lemma (4.1), we have $S_f-\S_i S_{p_i}^*S_{q_i}\in \LL^{n,\oo}$ and
$[S^*_f,S_f]-\[\S_i S_{q_i}^*S_{p_i},\S_i S_{p_i}^*S_{q_i}\]\in\LL^{n/2,\oo}$. For any $c\in S_1,$ the symbol map in \cite[Theorem 3.12]{U2} is given by
$$(\ls_1 S_f)(c)=(1_c\xt 1_c)T_{p_c}^*T_{q_c}(1_c\xt 1_c)=(p_c|q_c)_{S_c}(1_c\xt 1_c),$$
and
$$\ls_1\(\S_i S_{p_i}^* S_{q_i}\)(c)=\S_i\o{p_i(c)}q_i(c)(1_c\xt 1_c).$$
With Lemma \er{zz} it follows that
$$\h f|_{S_1}=\S_i\o p_i q_i.$$
Since $U S_{p_i}^* S_{q_i}U^*=\t B+\t T_{p_i}^*\t T_{q_i}=\t B+\t T_{\o p_i q_i},$ it follows that
$$U S_f U^*=\t T_{\S_i\o p_i q_i}+B=\t T_{\h f}+B$$
for $B\in\LL^{n,\oo}$. Therefore $U S_f U^*-\t T_{\h f}=U S_f U^*-\S_i\t T_{\o p_i q_i}\in\LL^{n,\oo}$ and
\begin{eqnarray*}&&U\[S^*_f,S_f\]U^*-\[\t T^*_{\h f},\t T_{\h f}\]\\&=&U\(\[S^*_f,S_f\]-\[\S_i S_{q_i}^* S_{p_i},\S_i S_{p_i}^*S_{q_i}\]\)U^*-\(U\[\S_i S_{q_i}^*S_{p_i},\S_i S_{p_i}^*S_{q_i}\]U^*-\[\t T^*_{\h f},\t T_{\h f}\]\)\in\LL^{n/2,\oo}.\end{eqnarray*}
\end{proof}
The explicit computation of the Dixmier trace uses the results of \cite{EZ} on strictly pseudo-convex domains. Let $\lS=\{z\in Z:\mbox{rank}(z)=1\}.$ Then $S_1\ic\lS$ has the defining function $r(z)=(z|z)-1.$ Therefore the contact $1$-form
$\lh=(\dl r-\o\dl r)/(2i)$ on $S_1$ \cite [Section 2.1]{EZ} is given by
$$\lh_c w=\f{(w|c)-(c|w)}{2i}$$
for all $c\in S_1$ and $w\in T_c(\lS)\ic Z.$ It follows that
$$(d\lh)_c(w_1,w_2)=\f{(w_1|w_2)-(w_2|w_1)}i.$$
Since $T_c(\lS)=Z^2_c\oplus Z^1_c=\Cl\,c\oplus Z^1_c$ we may write $w=i\la\,c+v,$ with $\la\in\Cl$ and $v\in Z_c^1.$ Then
$\lh_c(i\la\,c+v)=\la.$ It follows that $Ker(\lh_c)=Z^1_c$ and the Reeb vector field $E_\perp$ \cite[p. 614]{EZ} is given by $c\mapsto ic.$ Restricted to the tangent space $T_c(S_1)=i\Rl\,c\oplus Z^1_c,$ the $2$-form $d\lh$ has the radical
$i\Rl\,c=\Rl\,E_\perp$ and is non-degenerate on $Ker(\lh_c).$ Every $\lq\in\CL^\oo(S_1)$ defines a vector field $Z_\lq\in Ker(\lh)$ such that
$$d\lh(X,Z_\lq)=X\,\lq$$
for all vector fields $X\in Ker(\lh).$ For $\lf,\lq\in\CL^\oo(S_1)$ we obtain the {\bf boundary Poisson bracket}
$$\{\lf,\lq\}_\flat=d\lh(Z_\lf,Z_\lq)=Z_\lf\,\lq.$$
\begin{theorem}\label{b} Let $f_j,g_j\in\PL(Z\xx\o Z).$ Then
$$tr_\lo[S_{f_1},S_{g_1}]\cdots[S_{f_n},S_{g_n}]=C\,\I_{S_1}ds\P_{j=1}^n\{\h f_j,\h g_j\}_\flat,$$
where $ds$ is the normalized $K$-invariant measure, $\h f$ is the Poisson extension of $f$ and $\{\lf,\lq\}_\flat$ denotes the boundary Poisson bracket. The constant
$$C=\f1{(2\lp i)^n}\I_{S_1}\lh\yi\f{(d\lh)^{n-1}}{n!}$$
will be computed in the following Proposition \er{gam}.
\end{theorem}
\begin{proof} In general, if $T_1\in\LL^{n/2,\oo}$ and $T_2,\ldots,T_n\in\LL^{n,\oo}$ then $T_1T_2\cdots T_n\in\LL^1$ since $\LL^{k,\oo}\ic\LL^{k+\Le}$ for any $\Le>0.$ By Proposition \er{df} it follows that $U[S_{f_1},S_{g_1}]\cdots[S_{f_n},S_{g_n}]U^*-T\in\LL^1,$ where
$$T:=[\t T_{\h f_1},\t T_{\h g_1}]\cdots[\t T_{\h f_n},\t T_{\h g_n}]$$
is a generalized Toeplitz operator on $H^2(S_1)$ of order $-n.$ Applying \cite[Theorem 3]{EZ} it follows that
$$Tr_\lo[S_{f_1},S_{g_1}]\cdots[S_{f_n},S_{g_n}]=Tr_\lo(T)=\f1{(2\lp)^n}\I_{S_1}\lh\yi\f{(d\lh)^{n-1}}{n!}\,\ls_{-n}(T)(x,\lh_x),$$
where $\lh$ is the contact form. By \cite[Section 4]{EZ}, $T$ has the symbol
$$\ls_{-n}(T)=\P_{j=1}^n\,\ls_{-1}[\t T_{\h f_j},\t T_{\h g_j}]=\P_{j=1}^n\,\f1i\{\ls_0\t T_{\h f_j},\ls_0\t T_{\h g_j}\}_\lS
=\P_{j=1}^n\,\f1i\{\h f_j^{(0)},\h g_j^{(0)}\}_\lS$$
in terms of the Poisson bracket of $\lS.$ Here $\lf^{(0)}(t\,c)=\lf(c)$ denotes the $0$-homogeneous extension of
$\lf\in\CL^\oo(S_1).$ Now the assertion follows, since by \cite[Corollary 8]{EZ} we have for $t=1$
$$\f1i\{\lf^{(0)},\lq^{(0)}\}_\lS=Z_\lf\lq=\{\lf,\lq\}_\flat.$$
\end{proof}
Let $V=Z^1_{e_1}$ be the Peirce $1$-space for the minimal tripotent $e_1.$ If $a\neq 2$ or $r=1,$ then $V$ is an irreducible hermitian Jordan triple. If $a=2$ and $r>1$ then $Z=\Cl^{r\xx(r+b)}$ and
$V=\Cl^{(r-1)\xx 1}\oplus\Cl^{1\xx(r+b-1)}$ is a direct sum of two hermitian Jordan triples of rank $1.$ For any irreducible hermitian Jordan triple $V$ let $\lG_V$ denote the Gindikin $\lG$-function for the radial cone $\lO_V\ic V$ \cite{FK}. Let $r_V,\,d_V,\,p_V$ denote the rank, dimension and genus of $V,$ resp.
\begin{proposition}\label{gam} If $a\neq 2$ or $r=1,$ we have
$$\f1{(2\lp)^n}\I_{S_1}\lh\yi\f{(d\lh)^{n-1}}{(n-1)!}=\f{\lG_V(p_V-\f{n-1}{r_V})}{\lG_V(p_V)};$$
If $a=2$ we have
$$\f1{(2\lp)^n}\I_{S_1}\lh\yi\f{(d\lh)^{n-1}}{(n-1)!}=\f{1}{\lG(r)\,\lG(r+b)}.$$
In the rank $r=1$ case, where $Z=\Cl^d$ and $S_1=\Sl^{2n-1},$ we have $n=d=1+b$ and both formulas imply
$$\f1{(2\lp)^n}\I_{S_1}\lh\yi(d\lh)^{n-1}=1.$$
\end{proposition}
\begin{proof} Any irreducible hermitian Jordan triple $V$ has a 'quasi-determinant' function $\lD_V(u,v)$ such that the invariant measure on its conformal compactification $M,$ containing $V$ as an open dense subset of full measure, is a multiple of $\lD_V(v,-v)^{-p_V}\,d\ll(v),$ where $d\ll(v)$ is Lebesgue measure for the normalized inner product. Moreover, by \cite{EU} we have the polar integration formula
\begin{eqnarray}\label{vol}\I_V\f{d\ll(z)}{\lp^{d_V}}\lD(z,-z)^{-p_V}=\f{\lG_V(p_V-\f{d_V}{r_V})}{\lG_V(p_V)}.\end{eqnarray}
Let $M$ denote the compact complex manifold of all Peirce $2$-spaces $U\ic Z$ of rank $1.$ There is a canonical map
$$\lp:\lS\to M$$
which maps $z\in\lS$ onto its Peirce $2$-space $Z^2_z.$ In this way, $\lS$ becomes a hermitian holomorphic line bundle over $M$ which can be identified with the tautological line bundle $\LL=\bigcup_{U\in M}U.$ The subset $S_1\ic\lS$ corresponds to the circle bundle $\bigcup_{U\in M}S_U,$ where $S_U\al\Sl^1$ is the Shilov boundary of $U\in M.$ The holomorphic map $\lp$ satisfies
$$\mbox{ker}(d_c\lp)=Z^2_c,$$
since $Z^2_c\ic\mbox{ker}(d_c\lp)$ and both spaces are $1$-dimensional. Therefore $d\lh$ vanishes on
$$T_c(S_1)\ui\mbox{ker}(d_c\lp)=i\Rl\cdot c.$$
As a consequence there exists a $K$-invariant $2$-form $\lT$ on $M$ such that $d\lh=\lp^*\lT.$ Now $\lh,$ restricted to $S_U,$ is the usual contact form on $\Sl^1$ of volume $2\lp.$ It follows that
$$\f1{(2\lp)^n}\I_{S_1}\lh\yi\f{(d\lh)^{n-1}}{(n-1)!}=\f1{(2\lp)^n}\I_{S_1}\lh\yi\f{\lp^*\lT^{n-1}}{(n-1)!}
=\f1{(2\lp)^{n-1}}\I_{M}\f{\lT^{n-1}}{(n-1)!}.$$
In order to compute this integral, let $V=Z^1_{e_1}.$ A local coordinate for $M$ is given by the map $\ls:=\lp\oc\Lt:V\to M,$ where $\Lt:V\to\lS$ is defined by
$$\Lt(v)=e_1+v+\{ve_1^*v\}.$$
The semi-simple part $K'$ of $K$ acts transitively on $M,$ and induces a 'Moebius-type' biholomorphic action on $V$ such that $\ls$ becomes $K'$-equivariant. We have
$$\Lt^*d\lh=\Lt^*(\lp^*\lT)=\ls^*\lT.$$
Since $(d_0\Lt)v=v$ at the origin $0\in V$, the pull-back $\Lt^*(d\lh)|_v(v_1,v_2)=(d\lh)_{\Lt(v)}((d_v\Lt)v_1,(d_v\Lt)v_2)$ satisfies
$$\ls^*\lT|_0(v_1,v_2)=\Lt^*(d\lh)|_0(v_1,v_2)=\f{(v_1|v_2)-(v_2|v_1)}{i}.$$
Using complex coordinates $v_j$ with respect to an orthonormal basis of $V=T_0(V)$ this means
$$\ls^*\lT|_0=\S_{j=1}^{n-1}\f{d\o v_j\yi d v_j}{i}.$$
Now assume that $a\neq 2$ or $r=1.$ Then $M$ is irreducible. Since $\lT$ is invariant under $K,$ it follows that
$\ls^*\lT$ is invariant under the Moebius action. Since $d_V=n-1,$ we obtain for the volume form
$$\f{\ls^*\lT^{n-1}}{(n-1)!}=C\cdot\lD_V(v,-v)^{-p_V}d\ll(v),$$
where $C$ is a constant. Evaluating at $0\in V$ and using
\begin{eqnarray}\label{mea}d\ll(v)=\P_{j=1}^{n-1}\f{d\o v_j\yi d v_j}{2i}=\f1{(n-1)!}\(\S_{j=1}^{n-1}\f{d\o v_j\yi d v_j}{2i}\)^{n-1}
=\f1{2^{n-1}}\f{(\ls^*\lT|_0)^{n-1}}{(n-1)!}.\end{eqnarray}
it follows that $C=2^{n-1}.$ Since $\ls$ is a Zariski dense open embedding of full measure we obtain
$$\f1{(2\lp)^{n-1}}\I_{M}\f{\lT^{n-1}}{(n-1)!}=\f1{(2\lp)^{n-1}}\I_V\f{\ls^*\lT^{n-1}}{(n-1)!}=\f{C}{2^{n-1}}\I_V\f{d\ll(v)}{\lp^{n-1}}\,\lD_V(v,-v)^{-p_V}=\f{\lG_V(p_V-\f{n-1}{r_V})}{\lG_V(p_V)}$$
by applying \er{vol} to the irreducible hermitian Jordan triple $V=Z^1_{e_1}.$
Now assume $a=2$ and $r>1.$ Then $Z=\Cl^{r\xx(r+b)}$ and $M$ is reducible. More precisely,
$$\lS=\{z\in\Cl^{r\xx(r+b)}:\,rank(z)=1\}=\{\lx_1\lx_2:\,0\neq\lx_1\in\Cl^{r\xx 1},\,0\neq\lx_2\in\Cl^{1\xx(r+b)}\}.$$
Consider the associated projective spaces $M_1=\Pl(\Cl^{r\xx 1})=\{[\lx_1]:\,0\neq\lx_1\in\Cl^{r\xx 1}\}\al\Pl^{r-1}$ and
$M_2=\Pl(\Cl^{1\xx(r+b)})=\{[\lx_2]:\,0\neq\lx_2\in\Cl^{1\xx(r+b)}\}\al\Pl^{r+b-1}.$ Then $M=M_1\xx M_2$ is a direct product via the identification $([\lx_1],[\lx_2])\mapsto Ran(\lx_1\lx_2).$ For $i\in\{1,2\},$ the map $\lp_i:\lS\to M_i$ given by $\lx_1\lx_2\mapsto[\lx_i]$ is well-defined and the canonical map $\lp:\lS\to M_1\xx M_2$ is a product
$$\lp(\lx_1\lx_2)=([\lx_1],[\lx_2])=(\lp_1(\lx_1\lx_2),\lp_2(\lx_1\lx_2)).$$
Now $\lT=\lT_1\oplus\lT_2$ is the direct sum of $K'$-invariant $2$-forms $\lT_i$ on $M_i.$ Using the binomial theorem for (commuting) $2$-forms, the corresponding volume form is
$$\f{\lT^{n-1}}{(n-1)!}=\f{\lT_1^{n_1}}{n_1!}\yi\f{\lT_2^{n_2}}{n_2!}$$
for the dimensions $n_1=r-1,n_2=r+b-1$ adding up to $n_1+n_2=2(r-1)+b=n-1.$ It follows that
$$\f1{(2\lp)^{n-1}}\I_{M}\f{\lT^{n-1}}{(n-1)!}=\f1{(2\lp)^{n_1}}\I_{M_1}\f{\lT_1^{n_1}}{n_1!}\,\f1{(2\lp)^{n_2}}\I_{M_2}\f{\lT_2^{n_2}}{n_2!}.$$
In order to compute these integrals, put $V_1:=\Cl^{(r-1)\xx 1}$ and $V_2=\Cl^{1\xx(r+b-1)}.$ Then
$$V=Z^1_{e_1}=\{\bb0{v_2}{v_1}0:\,v_i\in V_i\}\al V_1\xx V_2.$$
The local coordinate $\ls(v_1,v_2)=(\ls_1(v_1),\ls_2(v_2))$ is of product type, where $\ls_i(v_i):=[1,v_i].$ In fact, putting
$v=\bb0{v_2}{v_1}0\in V,$ we obtain
$$\Lt(v)=e_1+v+\{ve_1^*v\}=\bb{1}{v_2}{v_1}{v_1v_2}=\ba{1}{v_1}\,\ab{1}{v_2}$$
and hence
$$\ls(v)=\lp(\Lt(v))=[\ba{1}{v_1}],[\ab{1}{v_2}].$$
The semi-simple part $K'=SU(r)\xx SU(r+b)$ of $K=S(U(r)\xx U(r+b))$ acts transitively on each factor $M_i$ and induces a 'Moebius-type' biholomorphic action on $V_i$ such that $\ls_i$ becomes $K'$-equivariant. Since $\lT_i$ is invariant under $K',$ it follows that $\ls_i^*\lT_i$ is invariant under this Moebius action. This implies for the volume form
$$\f{\ls_i^*\lT_i^{n_i}}{n_i!}=C_i\cdot(1+(v_i|v_i))^{-1-n_i}\,\,d\ll_i(v_i),$$
where $C_i$ is a constant. Using the relation
$$d\ll_i(v_i)=\f1{2^{n_i}}\f{(\ls_i^*\lT_i|_0)^{n_i}}{n_i!}$$
analogous to \er{mea}, it follows that $C_i=2^{n_i}.$ Since $\ls_i:V_i\to M_i$ is a Zariski dense open embedding of full measure we obtain
$$\f1{(2\lp)^{n_i}}\I_{M_i}\f{\lT_i^{n_i}}{n_i!}=\f1{(2\lp)^{n_i}}\I_{V_i}\f{\ls_i^*\lT_i^{n_i}}{n_i!}
=\f{C_i}{2^{n_i}}\I_{V_i}\f{d\ll_i(v_i)}{\lp^{n_i}}\,(1+(v_i|v_i)^{-1-n_i}=\f{\lG(1)}{\lG(1+n_i)}=\f1{n_i!}$$
by applying \er{vol} to the irreducible hermitian Jordan triple $V_i.$
\end{proof}

Finally, let us mention a relation involving numerical invariants of the domain $D$ and $V=Z^1_{e_1}$, which would make the formulas  more tractable.
\begin{lemma}\label{zv} Suppose that $a\neq 2.$ Then the rank $r_V$ and the genus $p_V$ of $V=Z^1_{e_1}$ satisfy the relation
$$r_V\,p_V=ra+b.$$
As a consequence, the $\lG$-function quotient in Proposition \er{gam} can also be expressed as
$$\f{\lG_V(p_V-\f{n-1}{r_V})}{\lG_V(p_V)}=\f{\lG_V(\f{a}{r_V})}{\lG_V(\f{ra+b}{r_V})}.$$
\end{lemma}
\begin{proof} We use the classification of hermitan Jordan triples \cite{L,N}. For $a=2,$ we obtain the Jordan triples
$Z=\Cl^{r\xx(r+b)}$ of type (I), for which the relation does not hold. The other cases are listed in the following table
\begin{center}
\begin{tabular}{|c|c|c|c|c|c|c|c|}
\hline
type &   Z           & rank  & a & b  & V & $r_V$  & $p_V$ \\ \hline
(II) & $\Cl^{(2r+\Le)\xx(2r+\Le)}_{\mbox{\tiny asym}} $ & r & 4  & $2\Le$ & $ \Cl^{2\xx(2(r-1)+\Le)}$ &2  & $2r+\Le$  \\ \hline
(III) & $\Cl^{r\xx r}_{\mbox{\tiny symm}} $ & r &1 & 0 & $ \Cl^{r-1}$ &1  & r  \\ \hline
(IV) & $\Cl^d_{\mbox{\tiny spin}} $ & 2 & d-2 & 0& $\Cl^{d-2}_{\mbox{\tiny spin}}$ &2  &d-2  \\ \hline
(V) & $\Ol^{1\xx 2}_\Cl $ & 2 & 6  & 4 & $ \Cl^{5\xx 5}_{\mbox{\tiny asym}}$ &2  & 8\\ \hline
(VI) & $\HL_3(\Ol)\otimes \Cl $ & 3 & 8  & 0 & $ \Ol^{1\xx 2}_\Cl$ &2  & 12 \\ \hline
\end{tabular}\end{center}
\end{proof}

\end{document}